\documentclass{amsart}

\usepackage{amsmath,amscd,amsfonts,amsthm,amssymb,amscd,mathrsfs}
\usepackage{appendix}
\usepackage{bm}
\usepackage{cite}
\usepackage{geometry}
\usepackage{graphicx}
\usepackage[colorlinks,linkcolor=blue,citecolor=red]{hyperref}
\hypersetup{CJKbookmarks}
\usepackage{setspace}
\usepackage{stmaryrd}
\usepackage{url}

\allowdisplaybreaks[4]

\newtheorem{definition}{Definition}[section]

\newtheorem{theorem}{Theorem}[section]
\newtheorem{lemma}[theorem]{Lemma}
\newtheorem{proposition}{Proposition}[section]
\newtheorem{corollary}{Corollary}[theorem]

\theoremstyle{definition}
\newtheorem{remark}{Remark}[section]

\numberwithin{equation}{section}

\DeclareMathOperator{\Tr}{Tr} 
\newcommand{\lHS}{\llbracket}
\newcommand{\rHS}{\rrbracket}
\newcommand{\Ft}{\widehat}
\newcommand{\Gp}{\mathbb{G}}
\newcommand{\DuGp}{\Ft{\mathbb{G}}}
\newcommand{\SU}{\mathrm{SU}}
\newcommand{\Op}{\mathbf{Op}}
\newcommand{\size[1]}{\langle{#1}\rangle}
\newcommand{\dist}{\mathrm{dist}}
\newcommand{\Hh[1]}{\mathcal{H}_{#1}}
\newcommand{\I[1]}{\mathrm{Id}_{\mathcal{H}_{#1}}}
\newcommand{\Df}{\mathbf{D}}

\author{Chengyang Shao}
\title{Toolbox of Para-differential Calculus on Compact Lie Groups}

\begin{document}
\maketitle
\begin{abstract}
This paper provides a toolbox of para-differential calculus on compact Lie groups. The toolbox is based on representation theory of compact Lie groups and contains exact formulas of symbolic calculus. Para-differential operators are constructed in a global, coordinate-free manner, giving lower order terms in symbolic calculus a clear form. The toolbox helps to understand non-local, nonlinear differential operators defined on certain manifolds with high symmetry.
\end{abstract}
\begin{spacing}{1.2}

\section{Introduction}
In this paper, we will develop a toolbox of harmonic analysis necessary for the study of motion of a water droplet under zero gravity. This is the starting point of a project proposed by the author in \cite{Shao2022} to understand nonlinear oscillation of water droplets. The paper is split from Section 1-4 of \cite{Shao2023}. The application of this toolbox to the Cauchy problem of oscillating water droplet can be found in \cite{Shao2023}.

\subsection{Spectral Localization and Dispersive PDEs on Compact Manifolds}
The study of nonlinear dispersive partial differential equations and harmonic analysis are closely linked to each other and mutually reinforcing. If discussed on a compact manifold, the geometry of the underlying space shall also come on the scene. We will be exploring an aspect of this topic in this paper. In order to extract the key ideas out of several important examples, we start our discussion by reviewing some previous results to illustrate the intersectionality of harmonic analysis on compact manifolds and nonlinear dispersive PDEs.

In the pioneering work \cite{BGT2004}, Burq, G\'{e}rad and Tzvetkov obtained Strichartz estimates for the unitary group $e^{it\Delta_g}$ on a compact Riemannian manifold $(M,g)$, and discussed its application to the Cauchy problem of semi-linear Schr\"{o}dinger equation on compact Riemannian manifold. The discussion relies on ``$L^p$ decoupling of Laplace eigenfunctions", discovered by Sogge \cite{Sogge1988}. As a result, they showed that on a compact Riemann surface $(M,g)$, the Cauchy problem of the cubic Schr\"{o}dinger equation 
$$
i\partial_tu+\Delta_gu=|u|^2u,
\quad
u(0)\in H^s(M)
$$
with $s>1/2$ is locally well-posed in $C^0_tH^s(M)$. They further showed in \cite{BGT2005} that the regime of local well-posedness can be extended to $s>1/4$ if $(M,g)$ is a Zoll surface (in particular, if it is the standard 2-sphere). A crucial harmonic analysis result needed by this is a bilinear eigenfunction estimate on Zoll surfaces, an inequality bounding the $L^2$ norm of the product of two eigenfunctions.

Hani continued the study in \cite{Hani2012}, showing that the Cauchy problem of the cubic Schr\"{o}dinger equation on $(M,g)$ is globally well-posed for $s>2/3$, regardless of the geometry. The proof is a modification of the ``I-method" due to Colliander-Keel-Staffilani-Takaoka-Tao (see \cite{CKSTT2002}). When transplanting the I-method from $\mathbb{R}^n$ or $\mathbb{T}^n$ to a general manifold, an immediate issue appears: one replaces Fourier analysis by spectral decomposition using Laplace eigenfunctions, but unlike Fourier modes, the product of two Laplace eigenfunctions is, in general, not an eigenfunction anymore. In order to overcome this difficulty, Hani established \emph{spectral licalization property} as a substitute. That is, the product of two eigenfunctions with eigenvalues $\lambda$ and $\mu$ respectively should sharply concentrate around the range $\sqrt{-\Delta_g}\in\big[|\sqrt{\lambda}-\sqrt{\mu}|,\sqrt{\lambda}+\sqrt{\mu}\big]$. This turns out to be an important ingredient for the proof.

Not surprisingly, the spectral localization property could be improved if certain geometric conditions are posed for the underlying manifold. Such improvement should then lead to stronger results for nonlinear dispersive PDEs. In fact, this is the case for compact Lie groups and homogeneous spaces. Highest weight theory for representations of a compact Lie group implies the following: the product of two Laplace eigenfunctions on a compact Lie group is a \emph{finite} linear combination of Laplace eigenfunctions. In \cite{BP2011} and \cite{BCP2015} (see also \cite{BBP2010}), precise form of this algebraic result plays a fundamental role for the search of periodic or quasi-periodic solutions of the semi-linear Schr\"{o}dinger equation
$$
i\partial_tu+\Delta u+\mu u=f(x,|u|^2)u
$$
on a compact Lie group or homogeneous space. Another example is \cite{DS2004}, where the underlying manifold is the standard sphere $\mathbb{S}^d$. The aforementioned spectral localization on Lie groups then becomes an algebraic property of spherical harmonic functions: the product of spherical harmonics of degree $p$ and $q$ is a linear combination of spherical harmonics of degree between $|p-q|$ and $p+q$. Using this fact, the authors casted a normal form reduction for semi-linear Klein-Gordon equations on $\mathbb{S}^d$.

Furthermore, the spectral localization property plays a fundamental role in the theory of pseudo-differential operators on compact Lie groups. Construction of pseudo-differential calculus on compact Lie groups depends very much on representation theory of compact Lie groups. Being the non-commutative analogy of Fourier theory on $\mathbb{R}^n$ or $\mathbb{T}^n$, spectral localization shows how matrix elements of different representations (``modes of different frequencies") interact. It thus occupies a similar part as the additive formula $e^{i\xi_1\cdot x}e^{i\xi_2\cdot x}=e^{i(\xi_1+\xi_2)\cdot x}$ does in commutative Fourier analysis and symbolic calculus.

Global symbolic calculus on compact Lie groups has been proposed by Ruzhansky-Turunen-Wirth \cite{RT2009}-\cite{RTW2014}, while the ideas seem to date back to Taylor \cite{Taylor1986}. However, as pointed out by Fischer \cite{Fis2015}, the pseudo-differential calculus in \cite{RT2009}-\cite{RTW2014} may lack rigorous justification. Fischer's work \cite{Fis2015} provides a mathematically rigorous and complete construction in which the symbolic calculus formulas are proved. In particular, Fischer proved that pseudo-differential operators of type $(1,0)$ defined via global symbols indeed coincides with the H\"{o}rmander class of $(1,0)$ pseudo-differential operators defined via local coordinates. To the author's knowledge, the first result on \emph{linear} dispersive PDEs on a compact Lie group with essential usage of global symbolic calculus is Bambusi-Langella \cite{BL2022}. The paper states a general frame work concerning ``abstract pseudo-differential operators" and obtained a growth estimate of the solution to \emph{linear} Schr\"{o}dinger equation with possibly time-dependent potential. Taking into account the pseudo-differential calculus constructed by \cite{Fis2015}, linear Schr\"{o}dinger equations on a compact Lie group thus realizes this abstract framework. 

\subsection{Para-differential Calculus on Manifolds} 
When coming to quasi-linear dispersive PDEs, however, some more technicalities beyond spectral localization will be necessarily required. Within our scope, we shall focus on one of them, namely para-differential calculus. 

Para-differential calculus originates from a systematic application of Littlewood-Paley decomposition to nonlinear PDEs. It first appeared in the form of \emph{para-product estimate}, leading to inequalities of the form $\|fg\|_{H^s}\lesssim |f|_{L^\infty}\|g\|_{H^s}+\|f\|_{H^s}|g|_{L^\infty}$ or composition estimates of the form $\|F(u)\|_{H^s}\lesssim \|u\|_{H^s}$. Such inequalities were crucial for local well-posedness of semi-linear dispersive PDEs; see for example, Saut-Temam \cite{SaTe1976} Linares-Ponce \cite{LP2014} for discussion of KdV equation, or Kato \cite{Kato1995} and Staffilani \cite{Staffilani1995} for discussion of semi-linear Schr\"{o}dinger equation. General estimates may also be found in \cite{Taylor2000}.

The defining character of such methodology was extracted as \emph{isolating the part with lowest regularity}. It has been extended to a class of operators called \emph{para-differential operators}; see Bony \cite{Bony1981} and Meyer \cite{Meyer1980}-\cite{Meyer1981}). Modelled on symbols $a(x,\xi)$ that are of limited regularity with respect to the spatial variable $x$, a symbolic calculus is still available for para-differential operators. Para-differential calculus thus becomes a powerful tool when usual pseudo-differential calculus is of limited usage due to low regularity.

Para-differential calculus has been successfully used in the study of quasi-linear PDEs (and also micro-local properties for fully nonlinear ones, see \cite{Bony1981} or H\"{o}rmander \cite{Hormander1997}; but we do not discuss that aspect). Taking fluid dynamics equations as an example, the gravity-capillary water waves equation was re-written into a para-differential form by Alazard-Burq-Zuilly \cite{ABZ2011}, and a local well-posedness result in the regime $s>3.5$ was deduced. A pleasant feature of this para-differential form is that the spectral properties of the dispersive relation for the problem are explicit. Not only is it an improvement of previous local well-posedness results (see for example \cite{Wu1997}-\cite{Wu1999}, \cite{Lannes2005}, \cite{CoSh2007}, \cite{ShZe2008}, \cite{MingZhang2009}), it also reduces the water waves equation into a form convenient for the study of long time behaviour. See the review \cite{IP2018} for detailed discussion, or \cite{BD2018} \cite{DIPP2017} \cite{IoPu2019} as examples of applying this para-differential form.

Being such a powerful tool, it is natural to ask whether para-differential calculus admits generalization to curved manifolds instead of $\mathbb{R}^n$ or $\mathbb{T}^n$. Let us briefly summarize previous attempts towards it. 

In \cite{KR2006}, Klainerman and Rodnianski extended the Littlewood-Paley theory to compact surfaces. They defined Littlewood-Paley decomposition of a tensor field via spectral cut-off of the tensorial Laplacian operator, studied the interaction of different frequencies, and deduced Bernstein type inequalities. This served as a toolbox in proving the $L^2$ bounded curvature conjecture for the Einstein equation. See \cite{KRS2015} for the details. The argument actually does not require para-differential calculus, but only an invariant version of Littlewood-Paley decomposition on compact surfaces. Not surprisingly, these properties can be reproduced and improved for compact Lie groups (see \cite{Fis2015}).

In \cite{BGdP2021}, Bonthonneau, Guillarmou and de Poyferr{\'e} developed a para-differential calculus for any compact manifold via local coordinate charts. This approach is a natural generalization of the H\"{o}rmander calculus on a curved manifold $M$: a $(1,1)$ symbol $a(x,\xi)$ is defined on the cotangent bundle of $M$, as a function that becomes a $(1,1)$ symbol under any local chart; para-differential operators corresponding to rough symbols are also defined by coordinate charts. A key feature is that symbolic calculus, just as the usual H\"{o}rmander calculus does, only preserves information at the order of principle symbol. The authors used this para-differential calculus to study microlocal regularity for an Anosov flow.

On the other hand, a global definition of para-differential operators on a compact manifold is introduced in Delort \cite{Delort2015}. Observing the commutator characterization of usual pseudo-differential operators (see e.g. Beals \cite{Beals1977}), Delort defined para-differential operators on a compact manifold via commutators with vector fields and Laplacian spectral projections. The most important result from this para-differential calculus is that, the commutator of para-differential operators of order $m$ and $m'$ is an operator of $m+m'-1$. With these preparations, Delort successfully obtained a normal form reduction for quasi-linear Hamiltonian Klein-Gordon equations on the standard sphere $\mathbb{S}^d$, and thus obtained extended lifespan estimate beyond the energy estimate regime.

Both constructions in \cite{BGdP2021} and \cite{Delort2015} share the common feature that formulas of symbolic calculus for para-differential operators involve only the \emph{principal symbol}, just as the usual H\"{o}rmander calculus on curved manifolds. Consequently, if one applies either one to a quasi-linear dispersive PDE, much of the spectral information related to the dispersive relation becomes implicit, although not lost. Besides, lower order information apart from the principal symbol is blurred. These disadvantages might cause extra difficulty in studying the long-time behavior of the solution.

\subsection{Organization of the Paper}
Our goal is to develop an analytic toolbox for PDEs on compact Lie groups and homogeneous spaces. In particular, the toolbox should enable us to reduce the spherical capillary water waves system in \cite{Shao2022} to a para-differential form convenient for the study of its long-time behaviour. Observing the advantages of incorporating geometric information, our goal will be accomplished by employing the notion of symbols in \cite{Fis2015} as the narrative formalism for para-differential calculus. 

Thus we transplant the sequence of ideas for para-differential calculus on $\mathbb{R}^n$ to a compact Lie group $\Gp$ (see for example Chapter 8-10 of \cite{Hormander1997}). The Fourier modes $e^{i\xi\cdot x}$ are replaced by irreducible unitary representations $\xi$, so that a symbol $a(x,\xi)$ is a field acting on the representation space of $\xi$. Due to unbounded degree of degeneracy of Laplace eigenfunctions (see the Weyl type lemma \ref{Weyl}), it is necessary to consider symbols as \emph{matrices} acting on representation spaces instead of \emph{scalar-valued functions}. 

We start by defining symbols and symbol classes on $\Gp$ in Section \ref{2}, then Littlewood-Paley decomposition in Section \ref{3}. Littlewood-Paley characterization for Sobolev and Zygmund space will be given here. The Stein theorem \ref{SteinTheorem} for pseudo-differential operators of type $(1,1)$ will be proved in Section \ref{3}. The spectral condition is introduced in Section \ref{3}, and Theorem \ref{Stein'} establishes the arbitrary $H^s$ boundedness for operators satisfying a spectral condition. Para-product estimates and Bony's para-linearization theorem are deduced as corollaries. In Section \ref{4}, para-differential operators are defined corresponding to rough symbols $a(x,\xi)$ with $C^r_*$ regularity in $x$. Formulas for symbolic calculus, i.e. composition and adjoint formulas are then proved for para-differential operators in Theorem \ref{Compo1}-\ref{ParaAdj}. That the commutator of para-differential operators of order $m$ and $m'$ is of $m+m'-1$ follows as a direct consequence.

None of the steps listed above are trivial. Let us discuss the major difficulties. As mentioned before, unlike Fourier modes on commutative groups, the product of two Laplace eigenfunctions on a non-commutative Lie group can \emph{never} be a Laplace eigenfunction. Furthermore, the dual of $\Gp$ does not form a group anymore, so the definition of ``differentiation" for a symbol $a(x,\xi)$ with respect to $\xi$ requires extra effort; see Section \ref{2} for the details. The spectral localization property for eigenfunctions on $\Gp$ will serve as the substitute for $e^{i\xi_1\cdot x}e^{i\xi_2\cdot x}=e^{i(\xi_1+\xi_2)\cdot x}$. In order to prove arbitrary $H^s$ boundedness for para-differential operators, we actually need a \emph{lower spectral bound} of the product. Corollary \ref{SpecPrd}, being stronger than Hani's general result \cite{Hani2012} in this special geometry, never appeared in previous literature on representation theory to the author's knowledge. It explicitly describes such bounds. The role it plays in establishing arbitrary $H^s$ boundedness for para-differential operators is clearly indicated in the proof of Theorem \ref{Stein'}. 

Additionally, functions of symbols thus cannot be manipulated so easily as in symbolic calculus on $\mathbb{R}^n$. We introduce the notion of \emph{quasi-homogeneous symbols} in Definition \ref{QuasiHomoSym}. They are essentially functions of classical differential symbols, hence still enjoys, to some extent, the good properties of the latter. It seems that this class of quasi-homogeneous symbols is broad enough to cover operators arising from PDEs that we are concerned with.

We emphasize that the algebraic structure of $\Gp$ is fundamental throughout our para-differential calculus. It enables us to write down the precise form of composition and adjoint operators. To this extent, this paper, together with the up-coming one (which will be the second split from \cite{Shao2023}), provide a novel formalism for non-local, \emph{quasi-linear} dispersive partial differential equations defined on certain manifolds with high symmetry. It can be expected that the formalism may be applied to other quasi-linear dispersive problems besides the spherical capillary waves equation.

\section{Preliminary: Representation of Compact Lie Groups}\label{2}
This section is devoted to preliminaries of representation theory of compact Lie groups. Notation will also be fixed in this section. Results from representation theory can be found in any standard textbook on representation theory, for example \cite{Procesi2007}. Concise descriptions can also be found in \cite{BP2011}.

\subsection{Representation and Spectral Properties}
To simplify our discussion, throughout this paper, we fix $\Gp$ to be a simply connected, compact, semisimple Lie group with dimension $n$ and rank $\varrho$, i.e. the dimension of any maximal torus in $\Gp$. By standard classification results of compact connected Lie groups, this is the essential part for our discussion. It is known from standard theory of Cartan subalgebra that all maximal tori in $\Gp$ are conjugate to each other, so the rank $\varrho$ is uniquely determined. The group $\Gp$ is in fact a direct product of groups of the following types: the special unitary groups $\SU(n)$, the compact symplectic groups $\mathrm{Sp}(n)=\mathrm{Sp}(2n;\mathbb{C})\cap\mathrm{U}(2n)$, the spin groups $\mathrm{Spin}(n)$, and the exceptional groups $G_2,F_4,E_6,E_7,E_8$.

We write $e$ for the identity element of $\Gp$. Let $\mathfrak{g}$ be the Lie algebra of $\Gp$, identified with the space of left-invariant vector fields on $\Gp$. If a basis $X_1,\cdots,X_n$ for $\mathfrak{g}$ is given, then for any multi-index $\alpha\in\mathbb{N}_0^n$, write 
\begin{equation}\label{NormalOrder}
X^\alpha=X_1^{\alpha_1}\cdots X_n^{\alpha_n}
\end{equation}
for the left-invariant differential operator composed by these basis vectors, in the ``normal" ordering. When $\alpha$ runs through all multi-indices, the operators $\{X^\alpha\}$ span the universal enveloping algebra of $\mathfrak{g}$, which is the content of the Poincaré–Birkhoff–Witt theorem; however, this fact is not needed within our scope. 

Since $\Gp$ is semi-simple and compact, the Killing form
$$
\Tr\big(\mathrm{Ad}(X)\circ\mathrm{Ad}(Y)\big),\quad X,Y\in\mathfrak{g}
$$ 
on $\Gp$ is non-degenerate and negative definite, where $\mathrm{Ad}$ is the adjoint representation. A Riemann metric on $\Gp$ can thus be introduced as the opposite of the Killing form. The Laplace operator corresponding to this metric, usually called \emph{Casimir element} in representation theory, is given by
$$
\Delta=\sum_{i=1}^n X_i^2,
$$
where $\{X_i\}$ is any orthonormal basis of $\mathfrak{g}$; the operator is independent from the choice of basis.

Let $\DuGp$ be the dual of $\Gp$, i.e. the set of equivalence classes of irreducible unitary representations of $\Gp$. For simplicity, we do not distinguish between an equivalence class $\xi\in\DuGp$ and a certain unitary representation that realizes it. The ambient space of $\xi\in\DuGp$ is an Hermite space $\Hh[\xi]$, with (complex) dimension $d_\xi$. If we equip the space $\Hh[\xi]$ with an orthonormal basis containing $d_\xi$ elements, then $\xi:\Gp\to\mathrm{U}(\Hh[\xi])$ can be equivalently viewed as a unitary-matrix-valued function $(\xi_{jk})_{j,k=1}^{d_\xi}$.

We shall equip the Lie group $\Gp$ with the normalized Haar measure. For simplicity, we will denote the integration with respect to this measure by $dx$. The \emph{Fourier transform} of $f\in\mathcal{D}'(\Gp)$ is defined by 
$$
\Ft f(\xi):=\int_\Gp f(x)\xi^*(x)dx\in\mathrm{End}(\Hh[\xi]).
$$
Conversely, for every field $a(\xi)$ on $\DuGp$ such that $a(\xi)\in\mathrm{End}(\Hh[\xi])$ for all $\xi\in\DuGp$, the \emph{Fourier inversion} of $a$ is defined by
$$
\check{a}(x)=\sum_{\xi\in\DuGp} d_\xi\Tr\big(a(\xi)\cdot\xi(x)\big),
$$
as long as the right-hand-side converges at least in the sense of distribution. 

Two different norms for $a(\xi)\in\mathrm{End}(\Hh[\xi])$ will be used in this paper, one being the operator norm
$$
\|a(\xi)\|:=\sup_{v\in\Hh[\xi],v\neq0}\frac{|a(\xi)v|}{|v|},
$$
one being the Hilbert-Schmidt norm
$$
\lHS a(\xi)\rHS:=\sqrt{\Tr(a(\xi)\cdot a^*(\xi))}.
$$
To simplify notation, when there is no risk of confusion, we shall omit the dependence of these norms on the representation $\xi$.

We fix the convolution on $\Gp$ to be \emph{right} convolution:
$$
(f*g)(x)=\int_\Gp f(y)g(y^{-1}x)dy.
$$
The only property not inherited from the usual convolution on Euclidean spaces is commutativity. Every other property, including the convolution-product duality, Young's inequality, is still valid. The only modification is that 
$$
\Ft{f*g}(\xi)=(\Ft g\cdot\Ft f)(\xi),
$$
which in general does not coincide with $(\Ft f\cdot\Ft g)(\xi)$ since the Fourier transform is now a matrix.

The Peter-Weyl theorem is fundamental for harmonic analysis on Lie groups:
\begin{theorem}[Peter-Weyl]\label{PeterWeyl}
Let $\mathcal{M}_\xi$ be the subspace in $L^2(\Gp)$ spanned by matrix entries of the representation $\xi\in\DuGp$. Then $\mathcal{M}_\xi$ is a bi-invariant subspace of $L^2(\Gp)$ of dimension $d_\xi^2$, and there is a Hilbert space decomposition
$$
L^2(\Gp)=\bigoplus_{\xi\in\DuGp}\mathcal{M}_\xi.
$$
If $f\in L^2(\Gp)$, the Fourier inversion
$$
(\Ft f)^\vee(x)=\sum_{\xi\in\DuGp} d_\xi\Tr\left(\Ft f(\xi)\cdot\xi(x)\right)
$$
converges to $f$ in the $L^2$ norm, and there holds the Plancherel identity
$$
\|f\|_{L^2(\Gp)}^2=\sum_{\xi\in\DuGp} d_\xi\big\lHS\Ft f(\xi)\big\rHS^2.
$$
\end{theorem}

In the following, standard constructions and results in highest weight theory will be directly cited. They are already collected in Berti-Procesi \cite{BP2011}.

Let $\mathfrak{t}\subset\mathfrak{g}$ be the Lie algebra of a maximal torus of $\Gp$, and write $\mathfrak{t}^*$ for its dual. Let $\{\bm{\alpha}_i\}_{i=1}^\varrho\subset\mathfrak{t}^*$ be the set of \emph{positive simple roots} of $\mathfrak{g}$ with respect to $\mathfrak{t}$. They form a basis for $\mathfrak{t}^*$. The \emph{fundamental weights} $\bm{\varpi}_1,\cdots,\bm{\varpi}_\varrho\in\mathfrak{t}^*$ admit several equivalent definitions, one of them being the unique set of vectors $\{\bm{\varpi}_i\}_{i=1}^\varrho$ satisfying 
$$
(\bm{\varpi}_i,\bm{\alpha}_j)=\frac{1}{2}\delta_{ij}|\bm{\alpha}_j|^2,
\quad
i,j=1,\cdots,\varrho.
$$
Here the inner product on the dual is inherited from the Killing form on $\mathfrak{g}$.

\begin{theorem}[Highest Weight Theorem]
Irreducible unitary representations of $\Gp$ are in 1-1 correspondence with the discrete cone
$$
\mathscr{L}^+:=\left\{\sum_{i=1}^\varrho n_i\bm{\varpi}_i,\, n_i\in\mathbb{N}_0\right\}.
$$
The correspondence assigns an irreducible unitary representation to its highest weight vector.
\end{theorem}
In particular, the trivial representation corresponds to the point $0\in\mathscr{L}^+$. From now on, we shall denote the assignment between representations $\xi\in \DuGp$ and highest weights $\bm{j}\in\mathscr{L}^+$ by $\bm{j}=\bm{J}(\xi)$, or $\xi=\Xi(\bm{j})$. The highest weight theorem implies a complete description of the spectrum of the Laplace operator $\Delta$ of $\Gp$.

\begin{theorem}\label{DeltaSpec}
Each space $\mathcal{M}_\xi$ is an eigenspace of the Laplace operator $\Delta$, with eigenvalue
$$
-|\bm{J}(\xi)+\bm{\varpi}|^2+|\bm{\varpi}|^2,
$$
where $\bm{\varpi}=\sum_{i=1}^\varrho \bm{\varpi}_i$. 
\end{theorem}
We write $\lambda_\xi=|\bm{J}(\xi)+\bm{\varpi}|^2-|\bm{\varpi}|^2$, so that $\{\lambda_\xi\}$ is the spectrum of the positive self-adjoint operator $-\Delta$. Imitating commutative harmonic analysis, we set the size functions $|\xi|:=\sqrt{|\lambda_\xi|}$, $\size[\xi]:=(1+\lambda_\xi)^{1/2}$. The operator $|\nabla|$, which is the square-root of $-\Delta$ defined via spectral analysis, will be frequently used. The values of $\size[\xi]$ are exactly the eigenvalues of $(1-\Delta)^{1/2}$. We will also use the following notation:
\begin{equation}\label{FourierSupp}
\mathfrak{S}[\gamma_1,\gamma_2]:=\left\{\xi\in\DuGp:\gamma_1\leq|\xi|\leq\gamma_2 \right\}.
\end{equation}

For $s\in\mathbb{R}$, the Sobolev space $H^s(\Gp)$ is defined to be the subspace of $f\in\mathcal{D}'(\Gp)$ such that $\|(1-\Delta)^{s/2}f\|_{L^2(\Gp)}<\infty$. By the Peter-Weyl theorem and the above characterization of Laplacian, the Sobolev norm is computed by 
$$
\|f\|_{H^s(\Gp)}^2=\sum_{\xi\in\DuGp} d_\xi\size[\xi]^{2s}\big\lHS\Ft f(\xi)\big\rHS^2.
$$

To proceed further, we need to introduce a notion of partial ordering on $\mathscr{L}^+$. Define the cone 
$$
\mathscr{R}^+:=\left\{\sum_{i=1}^\varrho n_i\bm{\alpha}_i,\, n_i\in\mathbb{N}_0\right\}.
$$
We then introduce the \emph{dominance order} on $\mathscr{L}^+$ as follows: for $\bm{j},\bm{k}\in\mathscr{L}^+$, we set $\bm{j}\prec \bm{k}$ if $\bm{j}\neq \bm{k}$ and $\bm{k}-\bm{j}\in\mathscr{R}^+$, and set $\bm{j}\preceq \bm{k}$ to include the possibility $\bm{j}=\bm{k}$. The relation $\prec$ defines a partial ordering on $\mathscr{L}^+$. With the aid of this root system, we obtain the following characterization of product of Laplace eigenfunctions (see, for example, \cite{Procesi2007}, Proposition 3 of page 345):

\begin{proposition}\label{SpecPrd0}
Let $u_{\bm{j}}\in\mathcal{M}_{\Xi(\bm{j})}$, $u_{\bm{k}}\in\mathcal{M}_{\Xi(\bm{k})}$ be Laplace eigenfunctions corresponding to highest weights $\bm{j},\bm{k}\in\mathscr{L}^+$ respectively. Then the product $u_{\bm{j}}u_{\bm{k}}$ is in the space
$$
\bigoplus_{\bm{l}\in\mathscr{L}^+:\,\bm{l}\preceq \bm{j}+\bm{k}}\mathcal{M}_{\Xi(\bm{l})}.
$$
\end{proposition}
We deduce from this characterization an important spectral localization property for products that will play a fundamental role in dyadic analysis. It is compatible with Hani's spectral localization \cite{Hani2012}: for a general compact Riemannian manifold, the product of eigenfunctions with eigenvalue $\mu,\lambda$ respectively should sharply concentrate in $\big[|\sqrt{\lambda}-\sqrt{\mu}|,\sqrt{\lambda}+\sqrt{\mu}\big]$ on the spectral side.
\begin{corollary}\label{SpecPrd}
There is a constant $0<c\leq1$, depending only on the algebraic structure of $\Gp$, with the following property. If $u_1\in\mathcal{M}_{\xi_1}$, $u_2\in\mathcal{M}_{\xi_2}$, then the product $u_1u_2\in\bigoplus_{\xi}\mathcal{M}_\xi$, where the range of sum is for
$$
\xi\in\mathfrak{S}\Big[c\big||\xi_1|-|\xi_2|\big|,|\xi_1|+|\xi_2|\Big].
$$
\end{corollary}
\begin{proof}
With a little abuse of notation, given a highest weight vector $\bm{j}\in\mathscr{L}^+$, we denote by $\bar{\bm{j}}$ the highest weight vector of $\Xi(\bm{j})^*$, the dual representation, which is also irreducible. It is known from representation theory that the highest weight of $\Xi(\bm{j})^*$ is just $-W_0(\bm{j})$, where $W_0$ is the element in the Weyl group with longest length. Since the Weyl group is orthogonal, there holds $|\bm{j}|=|\bar{\bm{j}}|$. By self-adjointness of $\Delta$, matrix entries of $\Xi(\bar{\bm{j}})$ correspond to the same eigenvalue of $\Delta$ as those of $\Xi(\bm{j})$ do.

Now suppose there are highest weights $\bm{j},\bm{k},\bm{l}\in\mathscr{L}^+$, and let $u_{\bm{j}}$, $u_{\bm{k}}$, $u_{\bm{l}}$ be linear combinations of matrix entries of $\Xi(\bm{j})$, $\Xi(\bm{k})$, $\Xi(\bm{l})$, respectively. From the above we know that e.g. $\bar{u}_{\bm{j}}\in\mathcal{M}_{\Xi(\bar{\bm{j}})}$. By Proposition \ref{SpecPrd0} and Schur orthogonality, the integral
$$
\int_\Gp u_{\bm{j}}u_{\bm{k}}\bar{u}_{\bm{l}} dx
$$
does not vanish only if the following relations are simultaneously satisfied:
$$
\bm{l}\preceq \bm{j}+\bm{k},
\quad \bm{j}\preceq \bar{\bm{k}}+\bm{l},
\quad \bar{\bm{k}}\preceq \bm{j}+\bar{\bm{l}}.
$$
Note that this is only a necessary condition. Denoting the pre-dual vector of $\bm{\alpha}_i$ by $\tilde{\bm{\alpha}}_i$ (i.e. $\tilde{\bm{\alpha}}_i\in\mathfrak{t}$, and $\bm{\alpha}_i(\tilde{\bm{\alpha}}_{i'})=\delta_{ii'}$), these relations are equivalent to
$$
\left.
\begin{aligned}
\bm{l}(\tilde{\bm{\alpha}}_i)&\leq \bm{j}(\tilde{\bm{\alpha}}_i)+\bm{k}(\tilde{\bm{\alpha}}_i)\\
\bm{j}(\tilde{\bm{\alpha}}_i)&\leq \bar{\bm{k}}(\tilde{\bm{\alpha}}_i)+\bm{l}(\tilde{\bm{\alpha}}_i)\\
\bar{\bm{k}}(\tilde{\bm{\alpha}}_i)&\leq \bm{j}(\tilde{\bm{\alpha}}_i)+\bar{\bm{l}}(\tilde{\bm{\alpha}}_i)\\
\end{aligned}
\right\}\quad i=1,\cdots,\varrho.
$$
Due to monotonicity of the eigenvalues with respect to the dominance order, the first inequality implies
$$
|\bm{l}+\bm{\varpi}|^2\leq|\bm{j}+\bm{k}+\bm{\varpi}|^2
$$
so we conclude that $|\bm{l}|\leq|\bm{j}|+|\bm{k}|$. Since $\{\bm{\alpha}_i\}_{i=1}^\varrho$ is a basis for $\mathfrak{t}^*$, it follows that $\{\tilde{\bm{\alpha}}_i\}_{i=1}^\varrho$ is a basis for $\mathfrak{t}$, so the second and third inequalities, together with $|\bm{l}|=|\bar{\bm{l}}|$, imply
$$
|\bm{j}-\bar{\bm{k}}|\lesssim |\bm{l}|.
$$
Using $|\bm{k}|=|\bar{\bm{k}}|$ once again, we obtain $\big||\bm{j}|-|\bm{k}|\big|\lesssim|\bm{l}|$. Thus the integral does not vanish only if 
$$
\big||\bm{j}|-|\bm{k}|\big|\lesssim|\bm{l}|\leq|\bm{j}|+|\bm{k}|.
$$
Using the formula for eigenvalues in Theorem \ref{DeltaSpec}, with $\bm{j}=\bm{J}(\xi_1)$, $\bm{k}=\bm{J}(\xi_2)$, $\bm{l}=\bm{J}(\xi)$, this is implies that the spectral projection of $u_{\bm{j}}u_{\bm{k}}$ to $\mathcal{M}_\xi$ is nontrivial only if
$$
\big||\xi_1|-|\xi_2|\big|
\lesssim |\xi|
\leq |\xi_1|+|\xi_2|.
$$
This is exactly what to prove.
\end{proof}
The dual representation of $\xi\in\DuGp$ is, in general, not equivalent to $\xi$, although they have the same dimension. This is why $\bar{\bm{j}}$ should be studied first. But the rank one case can be treated trivially: the special unitary group $\SU(2)$ is the only simply connected compact Lie group of rank 1. Irreducible unitary representations of $\SU(2)$ are uniquely labeled by their dimensions. For the Lie algebra $\mathfrak{su}(2)$ and the subalgebra corresponding to the subgroup $\mathrm{diag}(e^{i\psi},e^{-i\psi})$, the dominance order is equivalent to the usual ordering on natural numbers. Corollary \ref{SpecPrd} then becomes a well-known fact: the product of spherical harmonics of degree $p$ and $q$ is a linear combination of spherical harmonics of degree between $|p-q|$ and $p+q$, with coefficients given by the \emph{Clebsch–Gordan coefficients} (see e.g. \cite{Vilenkin1978}). This fact has been used by \cite{DS2004} as the substitute of spectral localization properties on spheres. On the other hand, the constant $c$ can be strictly less than 1. For example, the root system of $\SU(3)$ forms a hexagonal lattice, and $c=1/2$ in that case.

In estimating the magnitude of certain symbols, we will also be employing the following asymptotic formula of Weyl type. It is the discrete version of volume growth estimate for Euclidean balls:
\begin{lemma}\label{Weyl}
As $t\to\infty$,
$$
\sum_{\xi\in \mathfrak{S}[0,t]}d_\xi^2
\simeq
t^n.
$$
\end{lemma}
\begin{proof}
By Schur orthogonality, recalling $|\xi|=\sqrt{\lambda_\xi}$, we find that 
$$
\sum_{\xi\in \mathfrak{S}[0,t]}d_\xi^2
=\sum_{|\xi|\leq t}\dim\mathcal{M}_\xi
=\sum_{\lambda\in\sigma(-\Delta):\lambda\leq t^2}1,
$$
i.e. the number of eigenvalues of $-\Delta$ with magnitude less than $t^2$ (counting multiplicity). Using the Weyl law, we find that the ratio between
$$
\sum_{|\xi|\leq t}d_\xi^2
\quad\text{and}\quad
\frac{\omega_n}{(2\pi)^n}\mathrm{Vol}(\Gp)t^n
$$
has limit 1 as $t\to\infty$. Here $\omega_n$ is the volume of Euclidean $n$-unit ball.
\end{proof}

\subsection{Symbol and Symbolic Calculus}
A global notion of symbol can be defined on Lie groups due to their high symmetry. An explicit symbolic calculus was formally constructed by a series of works of Ruzhansky, Turunen, Wirth. Fischer's work \cite{Fis2015} provided a complete study of the object. We will closely follow \cite{Fis2015}.

We fix our compact Lie group $\Gp$ as in the previous subsection. The starting point will be a Taylor's formula on $\Gp$.
\begin{proposition}[Taylor expansion]\label{TaylorGp}
Let $q=(q_1,\cdots,q_M)$ be an $M$-tuple of smooth functions on $\Gp$, all vanishing at $e\in\Gp$, such that $(\nabla q_i)_{i=1}^M$ has rank $n$. For a multi-index $\alpha\in\mathbb{N}_0^m$, set $q^\alpha:=q_1^{\alpha_1}\cdots q_M^{\alpha_M}$. Corresponding to each multi-index $\alpha$, there is a left-invariant differential operator $X_q^{(\alpha)}$ of order $|\alpha|$ on $\Gp$, such that the following Taylor's formula holds for every smooth function $f$ on $\Gp$ and every $N\in\mathbb{N}_0$:
$$
f(xy)=\sum_{|\alpha|<N}q^\alpha(y^{-1})X_q^{(\alpha)}f(x)
+R_{N}(f;x,y).
$$
Here the remainder $R_{N}(f;x,y)$ depends linearly on derivatives of $f$ of order $\leq N$, is smooth in $x,y$, and satisfies $R_{N}(f;x,y)=O\big(|f|_{C^N}\dist(y,e)^N\big)$.
\end{proposition}

\begin{proof}[Sketch of proof]
Suppose without loss of generality that the first $n$ functions are of maximal rank at $e$. By the implicit function theorem, the mapping $x\to(q_1(x),\cdots,q_M(x))\in\mathbb{R}^M$ is a smooth embedding in some neighbourhood of $e\in\Gp$. Consider first $x=e$. Fix a basis $X_1,\cdots,X_n$ of $\mathfrak{g}$, such that $X_iq_j(e)=\delta_{ij}$ for $i,j\leq n$. For multi-indices $\alpha\in\mathbb{N}_0^M$, write $X^\alpha=X_1^{\alpha_1}\cdots X_n^{\alpha_n}$ (the ``normal" order); if the first $n$ components of $\alpha$ all vanish, then just set $X^\alpha=0$. Choose constants $c_\alpha$ inductively so that
$$
\left(\sum_{\alpha}c_{\alpha}X^\alpha\right) q^\beta(e)=\delta_{\alpha\beta},
\quad 
\beta_i\neq0\text{ only for }1\leq i\leq n.
$$
Then with $X^{(\alpha)}_q:=\sum_{\alpha}c_{\alpha}X^\alpha$, the formula holds for $x=e$ as an implication of the usual Taylor's formula. In order to pass to general $x\in \Gp$, just use the left translation operator $L_x$ to act on both sides, and notice that $X^{(\alpha)}_q$ are left-invariant differential operators.
\end{proof}
\begin{remark}
If $f$ takes value in a finite dimensional normed space $V$, the Taylor's formula given above is still valid. If we define the $C^N$ norm of $f$ by
$$
|f|_{C^N;V}:=\sup_{v^*\in V^*:v^*\neq0}\frac{\big|\langle f,v^*\rangle\big|_{C^N}}{|v^*|},
$$
then the estimate $|R_{N}(f;x,y)|_{V}=O\big(|f|_{C^N;V}\dist(y,e)^N\big)$ is valid, with the implicit constant independent from the dimension of $V$.
\end{remark}

Note that the differential operators $X_q^{(\alpha)}$ are defined independently for every multi-index, so the equality $X_q^{(\alpha+\beta)}=X_q^{(\alpha)}X_q^{(\beta)}$ is, in general, not valid.

A \emph{symbol} $a$ on $\Gp$ is simply a field $a$ defined on $\Gp\times\DuGp$, such that $a(x,\xi)$ is a distribution of value in $\mathrm{End}(\Hh[\xi])$ for each $\xi\in\DuGp$. If a basis for $\Hh[\xi]$ is chosen, the value $a(x,\xi)$ can be simply understood as a matrix function of size $d_\xi\times d_\xi$. We define the \emph{quantization} of a symbol $a$ formally by
\begin{equation}\label{Op(a)}
\Op(a)f(x):=\sum_{\xi\in\DuGp} d_\xi\Tr\left(a(x,\xi)\cdot\Ft f(\xi)\cdot\xi(x)\right).
\end{equation}
Conversely, if $A:C^\infty(\Gp)\to\mathcal{D}'(\Gp)$ is a continuous operator, then it is the quantization of the symbol
\begin{equation}\label{Op-Symbol}
\sigma[A](x,\xi):=\xi^*(x)\cdot(A\xi)(x).
\end{equation}
Here $A\xi$ is understood as entry-wise action. In this case the series (\ref{Op(a)}) then converges in $\mathcal{D}'(\Gp)$. The \emph{associated right convolution kernel} for $a(x,\xi)$ is defined by
\begin{equation}\label{Symbol-Ker}
\mathcal{K}(x,y):=\left(a(x,\cdot)\right)^\vee(y)
=\sum_{\xi\in\DuGp}d_\xi\Tr\left(a(x,\xi)\cdot\xi(y)\right),
\end{equation}
where the Fourier inversion is taken with respect to $\xi$. Formally, the action $\Op(a)f$ can be written as a convolution:
\begin{equation}\label{Ker-Conv}
f(\cdot)*\mathcal{K}(x,\cdot)
=\int_{\Gp}f(y)\mathcal{K}(x,y^{-1}x)dy.
\end{equation}

An intrinsic notion of difference operators acting on symbols was introduced by Fischer \cite{Fis2015}, generalizing the differential operator with respect to $\xi$ in harmonic analysis on $\mathbb{R}^n$. For any (continuous) unitary representation $(\tau,\mathcal{H}_\tau)$, Maschke's theorem ensures that $\tau=\oplus_j\xi_j$ for finitely many $\xi_j\in\DuGp$. A symbol $a(x,\xi)$ can be naturally extended to any $\tau$ by $a(x,\tau)=\oplus_ja(x,\xi_j)$, up to equivalence of representations. The definition of difference operators is then given by
\begin{definition}
Given any extended symbol $a(x,\cdot)$ and representation $\tau$, the difference operator $\Df_\tau$ gives rise to a new extended symbol in the following manner:
$$
\Df_\tau a(x,\pi):=a(x,\tau\otimes\pi)-a(x,\I[\tau]\otimes\pi).
$$
\end{definition}
For a tuple of representations $\boldsymbol{\tau}=(\tau_1,\cdots,\tau_p)$, write $\Df^{\boldsymbol{\tau}}=\Df_{\tau_1}\cdots\Df_{\tau_p}$, and $\Df^{\boldsymbol{\tau}}a(x,\xi)$ is then an endomorphism of $\mathcal{H}_{\tau_1}\otimes\cdots\otimes\mathcal{H}_{\tau_p}\otimes\Hh[\xi]$.

Corresponding to the functions $q=\{q_i\}_{i=1}^M$ as in Proposition \ref{TaylorGp}, Ruzhansky and Turunen introduced the so-called \emph{admissible difference operators}, also generalizing the differential operator with respect to $\xi$ in harmonic analysis on $\mathbb{R}^n$, but in a more computable way compared to the intrinsic definition. We employ the following definitions from \cite{RT2009}, \cite{RT2013} and \cite{RTW2014}:

\begin{definition}\label{RTAdm}
An $M$-tuple of smooth functions $q=(q_i)_{i=1}^M$ on $\Gp$ is said to be RT-admissible if they all vanishing at $e\in\Gp$ and $(\nabla q_i)_{i=1}^M$ has rank $n$. If in addition the only common zero of $(q_i)_{i=1}^M$ is the identity element, then the $M$-tuple is said to be strongly RT-admissible.
\end{definition}

\begin{definition}\label{Difference}
Given a function $q\in C^\infty(\Gp)$, the corresponding RT-difference operator $\Df_q$ acts on the Fourier transform of a $f\in\mathcal{D}'(\Gp)$ by
$$
\Big(\Df_q\Ft f\Big)(\xi):=\Ft{qf}(\xi).
$$
The corresponding collection of RT-difference operators corresponding to an $M$-tuple $q=(q_i)_{i=1}^M$ is the set of difference operators
$$
\Df_q^\alpha:=\Df_{q^\alpha}=\Df_{q_1}^{\alpha_1}\cdots\Df_{q_M}^{\alpha_M}.
$$
If the tuple $q$ is RT-admissible (strongly RT-admissible), the corresponding collection of RT difference operators is said to be RT-admissible (strongly RT-admissible).
\end{definition}

We write $\Df_{q,\xi}$ for the action of a difference operator on the $\xi$ variable. Formally, we have
\begin{equation}\label{Diff_qa}
(\Df_{q,\xi}a)(x,\xi)=\int_\Gp q(y)\left(\sum_{\eta\in\DuGp} d_\eta\Tr\left(a(x,\eta)\cdot\eta(y)\right)\right)\xi^*(y)dy,
\end{equation}
so $\Df_{q,\xi}$ commutes with any differential operator acting on $x$. To compare with $\mathbb{R}^n$, we simply notice that given a symbol $a(x,\xi)$ on $\mathbb{R}^n$, the Fourier inversion of $\partial_\xi a(x,\xi)$ with respect to $\xi$ is $iy a^{\vee}(x,y)$, i.e. multiplication by a polynomial. The functions $q$ on $\Gp$ then play the role of monomials on $\mathbb{R}^n$.

With the aid of difference operators, \cite{Fis2015}  introduced symbol classes of interest.
\begin{definition}\label{S^mrd}
Let $m\in\mathbb{R}$, $0\leq\delta\leq\rho\leq1$. Fix a basis $X_1,\cdots,X_n$ of $\mathfrak{g}$, and define $X^\alpha$ as in Proposition \ref{NormalOrder}. The symbol class $\mathscr{S}^m_{\rho,\delta}(\Gp)$ is the set of all symbols $a(x,\xi)$, such that $a(x,\xi)$ is smooth in $x\in\Gp$, and for any tuple of representation $\boldsymbol{\tau}=(\tau_1,\cdots,\tau_p)$ and any left-invariant differential operator $X^\alpha$, there is a constant $C_{\alpha\boldsymbol{\tau}}$ such that
$$
\big\|X^\alpha\Df^{\boldsymbol{\tau}}a(x,\xi)\big\|
\leq C_{\alpha\boldsymbol{\tau}}\size[\xi]^{m-\rho p+\delta|\alpha|}.
$$
Here the norm is taken to be the operator norm of $\mathcal{H}_{\tau_1}\otimes\cdots\otimes\mathcal{H}_{\tau_p}\otimes\Hh[\xi]$.
\end{definition}
\begin{definition}\label{S^mrdAdmi}
Let $m\in\mathbb{R}$, $0\leq\delta\leq\rho\leq1$. Fix a basis $X_1,\cdots,X_n$ of $\mathfrak{g}$, and define $X^\alpha$ as in Proposition \ref{NormalOrder}. Let $q=(q_1,\cdots,q_M)$ be an RT-admissible $M$-tuple. The symbol class $\mathscr{S}^m_{\rho,\delta}(\Gp;\Df_q)$ is the set of all symbols $a(x,\xi)$, such that $a(x,\xi)$ is smooth in $x\in\Gp$, and 
$$
\big\|X^{\alpha}_x\Df_{q,\xi}^{\beta}a(x,\xi)\big\|
\leq C_{\alpha}\size[\xi]^{m+\delta|\alpha|-\rho|\beta|}.
$$
We can also introduce the norms
\begin{equation}\label{RhoDeltaNorm}
\mathbf{M}_{k,\rho;l,\delta;q}^m(a):=\sup_{|\alpha|\leq k}\sup_{|\beta|\leq l}
\sup_{x,\xi}\size[\xi]^{\rho|\beta|-m-\delta|\alpha|}
\left\|X^\alpha_x\Df_{q;\xi}^\beta a(x,\xi)\right\|.
\end{equation}
In particular, for $\rho=\delta=1$, we write $\mathbf{M}_{k,l;q}^m(a)$ for simplicity.
\end{definition}

In \cite{Fis2015}, Fischer proved that if $q=(q_i)_{i=1}^M$ is a strongly RT-admissible tuple, then the symbol class in Definition \ref{S^mrdAdmi} does not depend on the choice of $q$, and in fact gives rise to the usual H\"{o}rmander class of pseudo-differential operators.

\begin{theorem}[Fischer \cite{Fis2015}, Theorem 5.9. and Corollary 8.13.]\label{Fischer}
\hfill\par
(1) Suppose $0\leq\delta\leq\rho\leq1$. If $q=(q_i)_{i=1}^M$ is a strongly RT-admissible tuple, then a symbol $a\in \mathscr{S}^m_{\rho,\delta}$ if and only if all the norms $\mathbf{M}_{r,\rho;l,\delta;q}^m(a)$ are finite.

(2) Moreover, if $\rho>\delta$ and $\rho\geq1-\delta$, then the operator class $\Op \mathscr{S}^m_{\rho,\delta}$ coincides with the H\"{o}rmander class $\Psi^m_{\rho,\delta}$ of $(\rho,\delta)$-pseudo-differential operators defined via local charts.
\end{theorem}
Let us briefly sketch Fischer's proof of (1) in Theorem \ref{Fischer}. The following lemma is a key ingredient, and is worthy of a specific mention:
\begin{lemma}[\cite{Fis2015}, Lemma 5.10]\label{Fisqq'Lem}
Suppose $q,q'\in\mathcal{D}(\Gp)$, such that $q'/q$ extends to a smooth function on $\Gp$. If a symbol $a=a(\xi)$ is such that $\|\Df_qa(\xi)\|$ does not grow faster than a power of $\size[\xi]$, then 
$$
\|\Df_{q'}a(\xi)\|\lesssim_{q,q'}\|\Df_qa(\xi)\|.
$$
\end{lemma}
Using this lemma, Fischer showed that for any two strongly RT-admissible tuple $q,q'$, the collection of norms $\mathbf{M}_{r,\rho;l,\delta;q}^m(a)$ and $\mathbf{M}_{r,\rho;l,\delta;q'}^m(a)$ are equivalent to each other. 

A convenient choice of strongly RT-admissible tuple is necessary for calculations. Fischer found that matrix entries of the fundamental representations produce a strongly RT-admissible tuple in the following way: for each unitary representation $\tau$ of $\Gp$, set
$$
q_{jk}^{(\tau)}(x)=\tau_{jk}(x)-\delta_{jk},
$$
i.e. the matrix entries of $\tau-\I[\tau]$. Fischer showed in Lemma 5.11. of \cite{Fis2015} that, as $\tau$ exhausts all of the fundamental representations $\Xi(\bm{\varpi}_i),i=1,\cdots,\varrho$, the collection $q=\{q_{jk}^{(\tau)}\}$ is strongly RT-admissible. The advantage of this choice is that it is essentially intrinsically defined for $\Gp$, and exploits the algebraic structure of $\Gp$. From now on, we will just defined the \emph{fundamental tuple of $\Gp$} as
\begin{equation}\label{QFund}
Q:=\bigcup_{\tau:\text{ fundamental representation}}\left\{\tau_{ij}-\delta_{ij}:i,j=1,\cdots,d_\tau\right\}.
\end{equation}
Thus, as Fischer pointed out, the symbol class $\mathscr{S}^m_{\rho,\delta}(\Gp;\Df_Q)$ in Definition \ref{S^mrdAdmi}, with $Q$ being the fundamental tuple (\ref{QFund}), in fact coincides with the intrinsic definition of symbols \ref{S^mrd}. This is because for each representation $\tau$, the matrix entries of $\Df_\tau a$ are exactly given by $\Df_{\tau_{jk}-\delta_{jk}}a$, so as $\tau$ exhausts all of the fundamental representations, the corresponding difference operators should give rise to the class $\mathscr{S}^m_{\rho,\delta}(\Gp)$. It is legitimate to write $\mathscr{S}^m_{\rho,\delta}(\Gp)$ and omit the dependence on strongly RT-admissible difference operators. This is how (1) of Theorem \ref{Fischer} was proved.

A particular property of the fundamental tuple $Q$ deserves a specific mention. For harmonic analysis on $\mathbb{R}^n$, differentiation in frequency usually simplifies the computation drastically. For analysis on compact group, the dual $\DuGp$ is discrete, whence no natural notion of ``continuous differentiation" is available. But the RT difference operators can be constructed in such a manner that they possess \emph{Leibniz type property}: given a tuple $q=(q_i)_{i=1}^M$, the corresponding RT difference operators are said to possess Leibniz type property, if there are constants $c_{j,k}^i\in\mathbb{C}$ such that for symbols $a,b$, 
\begin{equation}\label{Leibniz}
\Df_{q_i}(ab)
=\Df_{q_i}a\cdot b+a\cdot\Df_{q_i}b
+\sum_{j,k=1}^Mc_{j,k}^i\Df_{q_j}a\cdot \Df_{q_k}b.
\end{equation}
Inductively, this implies 
$$
\Df_{q}^\alpha(ab)
=\sum_{\substack{0\leq|\beta|,|\gamma|\leq|\alpha|
\\|\alpha|\leq|\beta|+|\gamma|\leq2|\alpha|}}
c_{\beta\gamma}^\alpha\Df_{q}^{\beta}a\cdot \Df_{q}^{\gamma}b.
$$
Condition (\ref{Leibniz}) is equivalent to
$$
q_i(xy)=q_i(x)+q_i(y)+\sum_{j,k=1}c_{j,k}^iq_j(x)q_k(y),
$$
so the matrix entries of $\tau-\I[\tau]$, where $\tau$ exhausts all fundamental representations of $\Gp$, is a good choice, because due to the equality
$$
\tau(xy)=\tau(x)\tau(y),
$$
they give rise to strongly admissible RT difference operators that satisfies Leibniz type property: for $q_{ij}(x)=\tau_{ij}(x)-\delta_{ij}$,
\begin{equation}\label{LeibnizFund}
q_{ij}(xy)=q_{ij}(x)+q_{ij}(y)+\sum_{k=1}^{d_\tau}q_{ik}(x)q_{kj}(y).
\end{equation}

The advantage of the symbol class $\mathscr{S}^m_{\rho,\delta}(\Gp)$ is that a symbolic calculus is available. The proof of the following results is quite parallel as in Euclidean spaces (not without technicality in verifying the necessary kernel properties):
\begin{theorem}[\cite{Fis2015}, Corollary 7.9.]\label{RegCompo}
Suppose $0\leq\delta<\rho\leq1$. If $a\in\mathscr{S}^m_{\rho,\delta}$, $b\in\mathscr{S}^{m'}_{\rho,\delta}$, then the composition $\Op(a)\circ\Op(b)\in\Op\mathscr{S}^{m+m'}_{\rho,\delta}$, and in fact if $q$ is any strongly RT-admissible tuple and $X^{(\alpha)}_q$ is as in Proposition \ref{TaylorGp}, then the symbol $\sigma$ of $\Op(a)\circ\Op(b)$ satisfies
$$
\sigma(x,\xi)-\sum_{|\alpha|< N}\left(\Df_{q,\xi}^\alpha a\cdot X_{q,x}^{(\alpha)}b\right)(x,\xi)
\in\mathscr{S}^{m+m'-(\rho-\delta)N}_{\rho,\delta}.
$$
One can thus write
$$
\sigma(x,\xi)\sim\sum_{\alpha}\left(\Df_{q,\xi}^\alpha a\cdot X_{q,x}^{(\alpha)}b\right)(x,\xi).
$$
\end{theorem}
\begin{theorem}[\cite{Fis2015}, Corollary 7.6.]\label{RegAdj}
Suppose $0\leq\delta<\rho\leq1$. If $a\in\mathscr{S}^m_{\rho,\delta}$, then the adjoint $\Op(a)^*\in\Op\mathscr{S}^{m}_{\rho,\delta}$, and in fact if $q$ is any strongly RT-admissible tuple and $X^{(\alpha)}_q$ is as in Proposition \ref{TaylorGp}, then the symbol $a^{\bullet;q}$ of $\Op(a)^*$ satisfies
$$
a^{\bullet;q}(x,\xi)-\sum_{|\alpha|< N}\left(\Df_{q,\xi}^\alpha X_{q,x}^{(\alpha)}a^*\right)(x,\xi)
\in\mathscr{S}^{m+m'-(\rho-\delta)N}_{\rho,\delta}.
$$
One can thus write
$$
a^{\bullet;q}(x,\xi)\sim\sum_{\alpha}\left(\Df_{q,\xi}^\alpha X_{q,x}^{(\alpha)}a^*\right)(x,\xi).
$$
\end{theorem}

Fischer also proved that for $0\leq\delta\leq\rho\leq1$ with $\delta<1$, if $a\in\mathscr{S}^m_{\rho,\delta}$, then $\Op(a)$ maps $H^{s+m}$ continuously to $H^s$ for any $s\in\mathbb{R}$. The proof is still parallel to that on Euclidean spaces. 

Finally, to prove (2) in Theorem \ref{Fischer}, i.e. $\Op\mathscr{S}^m_{\rho,\delta}(\Gp)=\Psi^m_{\rho,\delta}(\Gp)$, the H\"{o}rmander class of operators for $\rho>\delta$ and $\rho\geq1-\delta$, Fischer showed that the commutators of $\Op(a)$ with vector fields and $C^\infty$-multiplications satisfy Beals'
characterization of pseudo-differential operators as stated in \cite{Beals1977}. In particular, the operator class $\Op\mathscr{S}^m_{1,0}(\Gp)$ coincides with the H\"{o}rmander class $\Psi^m(\Gp)$ defined via local charts. Thus, it is possible to manipulate pseudo-differential operators on $\Gp$ using globally defined symbolic calculus, and all the lower-order information lost in operation with principal symbols can be retained.

\subsection{Order of a Symbol}
Unlike the case of $\mathbb{R}^n$, symbolic calculus on the non-commutative Lie group $\Gp$ involves endomorphisms of the representation spaces, hence suffers from non-commutativity. Thus, for example, properties of the commutator\footnote{Note that this is not the commutator of pseudo-differential operators.} $(ab-ba)(x,\xi)=:[a,b](x,\xi)$ of two symbols $a(x,\xi)$ and $b(x,\xi)$ is not as clear as on $\mathbb{R}^n$ (it simply vanishes for symbols on $\mathbb{R}^n$). With the aid of Fischer's theorem, however, we are able to show that the commutator of symbols of order $m$ and $m'$ respectively ``is reduced by order 1".

We introduce a formal definition of order as follows.
\begin{definition}\label{2Order}
Let $m\in\mathbb{R}$. We say that a symbol $a(x,\xi)$ on $\Gp$, regardless of regularity in $x$, is of order $m$, if for some strongly RT-admissible tuple $q$, there always holds
$$
\sup_{x\in\Gp}\big\|\Df_q^\beta a(x,\xi)\big\|\lesssim \size[\xi]^{m-|\beta|}.
$$
By Fischer's Lemma \ref{Fisqq'Lem}, the class does not depend on the choice of strongly RT -admissible tuple.
\end{definition}

Thus the class of symbols of order $m$, in our convention, is the collection of symbols that ``possess best decays upon differentiation in $\xi$". It necessarily includes all the $\mathscr{S}^m_{1,\delta}(\Gp)$ with $0\leq\delta\leq1$. The property that we need to prove is 
\begin{proposition}\label{2OrderComm}
Let $a,b$ be symbols on $\Gp$ of order $m$ and $m'$ and type 1 respectively. Then the commutator $[a,b]$ is of order $m+m'-1$. In particular, if $\delta\in[0,1]$, $a\in\mathscr{S}^m_{1,\delta}(\Gp)$ and $b\in\mathscr{S}^{m'}_{1,\delta}(\Gp)$, then $[a,b]\in\mathscr{S}^{m+m'-1}_{1,\delta}(\Gp)$.
\end{proposition}
\begin{proof}
It is quite alright to disregard the dependence on $x$ and consider Fourier multipliers $a=a(\xi)$, $b=b(\xi)$ alone, as no differentiation in $x$ is involved. By Fischer's theorem \ref{Fischer}, the operators $A=\Op(a)$ and $B=\Op(b)$ are in fact in the H\"{o}rmander class of pseudo-differential operators $\Psi^m_{1,0}$ and $\Psi^{m'}_{1,0}$ respectively. The H\"{o}rmander calculus on principal symbols (as functions on the cotangent bundle) then asserts that $[A,B]\in\Psi^{m+m'-1}_{1,0}$, with principal symbol given by the Poisson bracket. But $[A,B]$ is nothing but $[\Op(a),\Op(b)]=\Op\big([a,b]\big)$ as $a,b$ are assumed to be Fourier multipliers. Fischer's theorem again implies that $[a,b]\in\mathscr{S}^{m+m'-1}_{1,0}$.  To extend this to $x$-dependent symbols $a(x,\xi)$ and $b(x,\xi)$, it suffices to note that, for example, given a vector field $X$, the Leibniz rule $X_x\big([a,b]\big)=[X_xa,b]+[a,X_xb]$ holds. The operator norm of $X_x\big([a,b]\big)$ at a given representation $\xi$ can thus be estiamted similarly.
\end{proof}

\section{(1,1) Pseudo-differential Operators on Compact Lie Group}\label{3}
\subsection{Littlewood-Paley Decomposition}
We start with a continuous version of Littlewood-Paley decomposition. Such construction has already been explored by Klainerman and Rodnianski in \cite{KR2006} for a general compact manifold. Here we aim to expoit the Lie group structures to deduce more. Fix an even function $\phi\in C_0^\infty(\mathbb{R})$, such that $\phi(\lambda)=1$ for $|\lambda|\leq1/2$, and $\phi(\lambda)=0$ for $|\lambda|\geq1$. Setting $\psi(\lambda)=-\lambda\phi'(\lambda)$, we obtain a continuous partition of unity
$$
1=\phi(\lambda)+\int_1^\infty\psi\Big(\frac{\lambda}{t}\Big)\frac{dt}{t}.
$$
The \emph{continuous Littlewood-Paley decomposition} of a distribution $f\in\mathcal{D}'(\Gp)$ will then be defined by 
\begin{equation}\label{LPCont}
f=\phi\big(|\nabla|\big)f+\int_1^\infty\psi_t\big(|\nabla|\big)f\frac{dt}{t},
\quad 
\psi_t(\cdot)=\psi\left(\frac{\cdot}{t}\right)
\end{equation}
We also write the \emph{partial sum operator} as
$$
\phi_T(|\nabla|)f:=\phi\left(\frac{|\nabla|}{T}\right)f
=\phi\big(|\nabla|\big)f+\int_1^T\psi_t\big(|\nabla|\big)f\frac{dt}{t}.
$$
This is the convention employed by H\"{o}rmander \cite{Hormander1997}, Chapter 9. It is easy to deduce the Bernstein inequality:
\begin{proposition}\label{Bernstein}
If $f\in\mathcal{D}(\Gp)$ has Fourier support contained in $\mathfrak{S}[0,t]$, then for any $s\in\mathbb{R}$,
$$
\|f\|_{H^s}\lesssim t^s\|f\|_{L^2}.
$$
\end{proposition}

For each $t\geq1$, the integrand $\psi(|\nabla|/t)f$ is a smooth function on $\Gp$, whose Fourier support is contained in $\mathfrak{S}[t/2,t]$, where the definition is given in (\ref{FourierSupp}). In particular, the mode of $\psi(|\nabla|/t)f$ corresponding to zero eigenvalue is always zero. The integral (\ref{LPCont}) converges at least in the sense of distribution, and just as in the case of Littlewood-Paley decomposition on Euclidean space, the speed of convergence reflects the regularity property of $f$. Sometimes the \emph{discrete Littlewood-Paley decomposition} is also employed: setting $\vartheta(\lambda)=\phi_1(\lambda)-\phi_1(\lambda)$ and $\vartheta_j(\lambda)=\vartheta(\lambda/2^j)$, then 
\begin{equation}\label{LPDisc}
f=\phi\big(|\nabla|\big)f+\sum_{j\geq0}\vartheta_j\big(|\nabla|\big)f.
\end{equation}
The continuous and discrete Littlewood-Paley decomposition are parallel to each other, with summation in place of integration or vice versa, but in different scenarios one version can make the computation simpler than the other.

Since the integrand in (\ref{LPCont}) is the convolution of $f$ with kernel 
$$
\left(\psi_t(|\xi|)\I[\xi]\right)^\vee(x)
=\sum_{\xi\in\Gp}d_\xi\psi_t(|\xi|)\Tr\left(\xi(x)\right),
$$
it is thus important to understand a kernel with finite Fourier support and is merely a scaling on each representation space. Notice that the order of convolution does not really matter for this specific case. We thus state the following fundamental lemma concerning kernels of spectral cut-off operators:

\begin{lemma}\label{CutOffKer}
Let $h\in C_0^\infty(\mathbb{R})$. Let $P$ be any classical differential operator with smooth coefficients of degree $N\in\mathbb{N}_0$. Then for $t\geq1$, the function
$$
\check h_t:=\left(h_t(|\xi|)\I[\xi]\right)^\vee
=\sum_{\xi\in\Gp}d_\xi h\left(\frac{|\xi|}{t}\right)\Tr\left(\xi\right)
$$
satisfies
$$
\|P\check h_t\|_{L^1}\lesssim_{N}t^{N}|h|_{L^\infty}.
$$
If $f\in C(\Gp)$ is such that for some $r>0$, $f(x)=O(\dist(x,e)^r)$ when $x\to e$, then
$$
\|f\check h_t\|_{L^1}\lesssim |h|_{L^\infty}t^{-r}.
$$
\end{lemma}
With a little abuse of notation, we denote the Fourier inversion of the symbol $h_t(|\xi|)\I[\xi]$ as $\check{h}_t$. The proof can be found in \cite{Fis2015}. It in fact only depends on the compact manifold structure of $\Gp$, and follows from properties of heat kernel on a compact manifold. As a corollary, the symbol $h_t(|\xi|)\I[\xi]$ is of class $\mathscr{S}^0_{1,0}$ if $h$ is smooth and has compact support. We deduce a corollary that will be used later.

\begin{corollary}\label{ConvVanish}
Let $h\in C_0^\infty(\mathbb{R})$. If $f\in \mathcal{D}(\Gp)$ is such that $f(x)=O(\dist(x,e)^N)$ for some $N\in\mathbb{N}_0$, then $h_t*f$ remains in a bounded subset of $\mathcal{D}(\Gp)$, for $t\geq1$, and
$$
\big|(\check h_t*f)(x)\big|
\lesssim_N |h|_{L^\infty}t^{-N}\sum_{k=0}^N t^k|f|_{C^k}\dist(x,e)^k.
$$
\end{corollary}
\begin{proof}
For simplicity we prove for $N=1$; the general case follows similarly. The order of convolution does not matter since $h_t(|\xi|)\I[\xi]$ is a scaling on each $\Hh[\xi]$. We use Taylor's formula to write
$$
\begin{aligned}
(\check h_t*f)(x)
&=\int_{\Gp}\check h_t(y)f(y^{-1}x)dy\\
&=\int_{\Gp}\check h_t(y)f(y)dy+\int_{\Gp}\check h_t(y)R_1(f;y^{-1},x)dy.
\end{aligned}
$$
The first term is controlled by $|h|_{L^\infty}|f|_{C^0}t^{-1}$ by Lemma \ref{Ker-Conv}, and the second is controlled by $|h|_{L^\infty}\dist(x,e)|f|_{C^1}$, by Taylor's formula.
\end{proof}

A more general multiplier theorem was deduced in \cite{Fis2015} from Lemma \ref{CutOffKer}:
\begin{theorem}\label{Multm}
Let $h$ be a smooth function on $[0,\infty)$. Define the multiplier norm
$$
\|h\|_{m;k}=\sup_{\lambda\geq0,l\leq k}\left(1+|\lambda|\right)^{-m+l}\big|h^{(l)}(\lambda)\big|.
$$
If $q\in C^\infty(\Gp)$ vanishes of order $N-1$ at $e$, where $N\geq1$ is an integer, then with 
$$
a_t(\xi)= h_t(|\xi|)\I[\xi],
$$
the symbol $a_t$ is of class $\mathscr{S}^m_{1,0}$, and the difference norm is estimated by
$$
\|\Df_qa_t(\xi)\|\lesssim_{q} \|h\|_{m;k}\size[\xi]^{m-N}t^{-m}.
$$
\end{theorem}

The Littlewood-Paley characterization of Sobolev space is obtained immediately: 
\begin{proposition}\label{LPHs}
Given $s\in\mathbb{R}$, a distribution $f$ belongs to $H^s(\Gp)$ if and only if for some non-vanishing $h\in C_0^\infty(0,\infty)$, there holds
$$
\left\|\phi\big(|\nabla|\big)f\right\|_{L^2}^2
+\int_0^\infty t^{2s-1}\left\|h_t\big(|\nabla|\big)f\right\|_{L^2_x}^2dt
<\infty,
$$
where $h_t(\lambda)=h(\lambda/t)$. The lower bound of integral does not cause singularity since $h=0$ near 0. Similarly, distribution $f$ belongs to $H^s(\Gp)$ if and only if for some non-zero $h\in C_0^\infty(0,\infty)$, there holds
$$
\left\|\phi\big(|\nabla|\big)f\right\|_{L^2}^2
+\sum_{j\geq0}2^{2sj}\left\|h_{2^j}\big(|\nabla|\big)f\right\|_{L^2}^2.
$$
The square root of either of the above quadratic forms is equivalent to $\|f\|_{H^s}$.
\end{proposition}
From Lemma \ref{CutOffKer}, we also obtain a characterization of the Zygmund space $C^r_*(\Gp)$ with $r\geq0$:

\begin{proposition}\label{LPZyg}
For $r\geq0$, a distribution $f\in\mathcal{D}'(\Gp)$ is in the Zygmund class $C^r_*(\Gp)$ if and only if
$$
\sup|\phi(|\nabla|)f|+\sup_{t\geq1}t^r\left|\psi_t\big(|\nabla|\big)f\right|_{L^\infty}<\infty.
$$
This quantity is equivalent to the Zygmund space norm defined via local coordinate charts.
\end{proposition}
\begin{proof}[Sketch of proof]
The proof does not differ very much from the Euclidean case. Suppose first $0<r<1$ is not an integer, and $f\in C^r$. Then 
$$
\psi_t\big(|\nabla|\big)f(x)
=\int_{\Gp}\check\psi_t(y^{-1}x)f(y)dy
=\int_{\Gp}\check\psi_t(y^{-1}x)\left(f(y)-f(x)\right)dy.
$$
Note that the order of convolution does not matter, since on the Fourier side, $\psi_t\big(|\nabla|\big)$ is just a scaling on each representation space, and the last step is because $\check\psi_t$ always has mean zero. Then
$$
|\psi_t\big(|\nabla|\big)f|_{L^\infty}
\lesssim_r |f|_{C^r}\int_{\Gp}\big|\check\psi_t(y^{-1}x)\big|\cdot\dist(x,y)^rdy,
$$
and the right-hand-side is controlled by $|f|_{C^r}t^{-r}$ by Lemma \ref{CutOffKer}. As for the opposite direction, if we know that $\big|\psi_t\big(|\nabla|\big)f\big|_{L^\infty}\lesssim t^{-r}$, then write
$$
\begin{aligned}
\int_1^\infty\left(\psi_t\big(|\nabla|\big)f(x)-\psi_t\big(|\nabla|\big)f(y)\right)\frac{dt}{t}
=\int_{t\leq\dist(x,y)^{-1}}+\int_{t>\dist(x,y)^{-1}}.
\end{aligned}
$$
When $x$ is close to $y$, the first integral is estimated, using Lemma \ref{CutOffKer} again, by
$$
\sum_{|\beta|=1}\int_1^{\dist(x,y)^{-1}}\left|X^\beta\psi_t\big(|\nabla|\big)f\right|_{L^\infty}\cdot \dist(x,y)\frac{dt}{t}
\,\lesssim_r\,
\sup_{t\geq1}t^r\left|\psi_t\big(|\nabla|\big)f\right|_{L^\infty}\cdot\dist(x,y)^r.
$$
The second integral is simply estimated by
$$
2\sup_{t\geq1}t^r\left|\psi_t\big(|\nabla|\big)f\right|_{L^\infty}\int_{t>\dist(x,y)^{-1}}t^{-r-1}dt
\,\simeq_r\,
\sup_{t\geq1}t^r\left|\psi_t\big(|\nabla|\big)f\right|_{L^\infty}\cdot\dist(x,y)^r.
$$
This proves the equivalence for $0<r<1$. As for the general case, we can simply apply the classical Schauder theory and interpolation for H\"{o}lder spaces; recall that interpolation of H\"{o}lder spaces results in the Zygmund space when the intermediate index is an integer.
\end{proof}

With the aid of Lemma \ref{Ker-Conv}, we immediately obtain
\begin{corollary}\label{LPZygCor}
Suppose $r\geq0$ and $f\in C^r_*$. Fix a basis $X_1,\cdots,X_n$ of $\mathfrak{g}$, and define $X^\alpha$ as in Proposition \ref{NormalOrder}. Then there holds, for $t\geq1$,
$$
\left|X^\alpha\phi_t\big(|\nabla|\big)f\right|_{L^\infty}
\lesssim_{r,\alpha}
\left\{
\begin{aligned}
    & t^{(|\alpha|-r)_+}|f|_{C^r_*} & \quad |\alpha|\neq r \\
    & \log(1+t)|f|_{C^r_*} & \quad |\alpha|=r
\end{aligned}
\right.
$$
where $s_+=\max(s,0)$. Furthermore, for $r>0$, $\left|f-\phi_t\big(|\nabla|\big)f\right|_{L^\infty}\lesssim t^{-r}|f|_{C^r_*}$ when $t\geq1$.
\end{corollary}

\begin{remark}\label{Zygmund<0}
Just as in the Euclidean case, we can \emph{define} Zygmund spaces with index $<0$ by Littlewood-Paley characterization: we say that $f\in C^{r}_*$ for $r<0$, iff
$$
|f|_{C^{-r}_*}
:=\sup|\phi(|\nabla|)f|
+\sup_{t\geq1}t^{r}\left|\psi_t\big(|\nabla|\big)f\right|_{L^\infty}<\infty.
$$
This gives a united definition of Zygmund spaces on $\Gp$. In particular, $|\nabla|^s$ is continuous from $C^{r+s}_*$ to $C^r_*$ for all $r\in\mathbb{R}$. Furthermore, for $r<0$, there holds a growth estimate
$$
\begin{aligned}
\left|\phi_t\big(|\nabla|\big)f\right|
\leq \sup|\phi(|\nabla|)f|+\int_1^t\left|\psi_\tau\big(|\nabla|\big)f\right|d\tau
\lesssim t^{-r}|f|_{C^r_*}.
\end{aligned}
$$
\end{remark}

\begin{remark}\label{VectZyg}
These results easily generalize to vector-valued functions. If $r>0$, $V$ is a finite dimensional normed space, $f\in C^r_*(\Gp;V)$, then the norm $|f|_{C^r_*;V}$ is defined as
$$
\sup_{v^*\in V^*,v^*\neq0}\frac{|\langle v^*,f\rangle|_{C^r_*}}{|v^*|_{V^*}}.
$$
We have e.g.
$$
|f|_{C^r_*;V}\simeq\sup|f|_{L^\infty;V}+\sup_{x,y\in\Gp}\frac{|f(x)-f(y)|_V}{\dist(x,y)^r},\quad 0<r<1,
$$
where $|f|_{L^\infty;V}=\sup|f|_V$, and the equivalence is \emph{independent} from $V$. Repeating the same argument as in the proof of Proposition \ref{LPZyg}, just with $\langle v^*,f\rangle$ in place of $f$, then taking supremum over $v^*\in V^*$, we arrive at
$$
|f|_{C^r_*;V}\simeq_r |\phi(|\nabla|)f|_{L^\infty;V}+\sup_{t\geq1}t^r\left|\psi_t\big(|\nabla|\big)f\right|_{L^\infty;V},
$$
$$
\left|X^\alpha\phi_t\big(|\nabla|\big)f\right|_{L^\infty;V}
\lesssim_{r,\alpha}
\left\{
\begin{aligned}
    & t^{(|\alpha|-r)_+}|f|_{C^r_*;V} & \quad |\alpha|\neq r \\
    & \log(1+t)|f|_{C^r_*;V} & \quad |\alpha|=r
\end{aligned}
\right.
$$
and $\left|f-\phi_t\big(|\nabla|\big)f\right|_{L^\infty;V}\lesssim t^{-r}|f|_{C^r_*;V}$ for $t\geq1$. Results for Zygmund spaces of index $\leq0$ in Remark \ref{Zygmund<0} can be generalized to vector-valued functions similarly. All the implicit constants above \emph{do not depend on} $V$.
\end{remark}

It is quite legitimate to refer the decomposition (\ref{LPCont}) or (\ref{LPDisc}) as \emph{geometric Littlewood-Paley decomposition}. In \cite{KR2006}, Klainerman and Rodnianski commented that, unlike Euclidean case, the geometric Littlewood-Paley decomposition on a general manifold does not exclude high-low interactions, although this drawback does not affect applications within their scope. As Berti-Procesi \cite{BP2011} pointed out, this is because that on a general compact manifold, the product of Laplacian eigenfunctions does not enjoy good spectral property (unlike the Fourier modes $e^{ix\xi}$), and this might cause transmission of energy between arbitrary modes for nonlinear dispersive systems. However, when working on a compact Lie group, this issue is resolved by spectral localization properties, Proposition \ref{SpecPrd0} and Corollary \ref{SpecPrd}.

\subsection{(1,1) Pseudo-differential Operator}
Even in the Euclidean case, the symbol class $\mathscr{S}^m_{1,1}(\mathbb{R}^n)$ exhibits exotic properties compared to smaller classes $\mathscr{S}^m_{1,\delta}(\mathbb{R}^n)$ with $\delta<1$, and ``must remain forbidden fruit" as commented by Stein \cite{SteMur1993} (Subsection 1.2., Chapter 7). We thus cannot expect a satisfactory calculus for general $\mathscr{S}^m_{1,1}(\Gp)$ symbols. But in analogy to the Euclidean case, a series of theorems and constructions still remain valid for $\mathscr{S}^m_{1,1}(\Gp)$.

\begin{theorem}[Stein]\label{SteinTheorem}
Suppose $a\in \mathscr{S}^m_{1,1}(\Gp)$. Then for $s>0$, $\Op(a)$ is a bounded linear operator from $H^{s+m}$ to $H^s$. The effective version reads
$$
\|\Op(a)f\|_{H^s}\leq C_{s;q}\mathbf{M}^m_{[s]+1,n+2;q}(a)\|f\|_{H^{s+m}},
$$
where $q$ is a strongly RT-admissible tuple, and the norm $\mathbf{M}^m_{k,l;q}$ is defined in (\ref{RhoDeltaNorm}).
\end{theorem}

Before we proceed to the proof, let us first state several lemmas that will be used later.

\begin{lemma}[Taylor \cite{Taylor2000}, Lemma 4.2 of Chapter 2]\label{l2Seq}
For $s<1$, there holds
$$
\sum_{j\in\mathbb{Z}}2^{2sj}\left|\sum_{k:k<j}2^{k-j}a_k\right|^2
\leq C_s\sum_{k\in\mathbb{Z}}2^{2sk}|a_k|^2.
$$
For $s>0$, there holds
$$
\sum_{j\in\mathbb{Z}}2^{2sj}\left|\sum_{k:k\geq j}a_k\right|^2\leq C_s\sum_{k\in\mathbb{Z}}2^{2sk}|a_k|^2.
$$
\end{lemma}
As a corollary, we also have the following lemma:
\begin{lemma}\label{H^ss>0}
Let $s>0$, $l\geq s+1$ be an integer. Fix a basis $X_1,\cdots,X_n$ of $\mathfrak{g}$, and define $X^\alpha$ as in Proposition \ref{NormalOrder}. There is a constant $C=C(s,l)$ with the following properties. If $\{f_k\}_{k\geq0}$ is a sequence in $H^l(\Gp)$, such that for any multi-index $\alpha$ with $|\alpha|\leq l$, there holds
$$
\|X^\alpha f_k\|_{L^2}\leq 2^{k(|\alpha|-s)}c_k,\quad (c_k)\in \ell^2(\mathbb{N}),
$$
then in fact $\sum_k f_k\in H^s(\Gp)$, and 
$$
\left\|\sum_{k\geq0}f_k\right\|_{H^s}^2
\leq C\sum_{k\geq0}c_k^2.
$$
\end{lemma}
\begin{proof}
We will use the discrete Littlewood-Paley decomposition (\ref{LPDisc}) to conduct the proof. If $j\leq k$, then by taking $\alpha=0$, obviously
$$
\|\vartheta_j\big(|\nabla|\big)f_k\|_{L^2}\leq\|f_k\|_{L^2}\leq 2^{-ks}c_k.
$$
If $j>k$, then by Bernstein's inequality (Proposition \ref{Bernstein}),
$$
\begin{aligned}
\|\vartheta_j\big(|\nabla|\big)f_k\|_{L^2}
&\leq C{2^{-lj}}\|\vartheta_j\big(|\nabla|\big)f_k\|_{H^l}\\
&\simeq C2^{-lj}\sum_{|\alpha|=l}\|X^\alpha f_k\|_{L^2}\\
&\leq C2^{-ks}2^{l(k-j)}c_k.
\end{aligned}
$$
Thus, for $f=\sum_kf_k$, we have
$$
2^{js}\|\vartheta_j\big(|\nabla|\big) f\|_{L^2}
\leq C\sum_{k:k<j}2^{(l-s)(k-j)}c_k
+\sum_{k:k\geq j}2^{(j-k)s}c_k.
$$
Using Lemma \ref{l2Seq}, we find that $\left\{2^{js}\|\vartheta_j\big(|\nabla|\big) f\|_{L^2}\right\}_{j\geq0}$ is an $\ell^2$ sequence. Thus $f\in H^s$.
\end{proof}

\begin{proof}[Proof of Stein's Theorem]
The proof here is very much like the one in \cite{Met2008}, Section 4.3. Writing $a_0(x,\xi)=a(x,\xi)\phi(|\xi|)$, consider the dyadic decomposition
$$
a(x,\xi)=a_0(x,\xi)+\sum_{j=1}^\infty a_j(x,\xi),
\quad\text{with}\quad 
a_j(x,\xi)=a(x,\xi)\vartheta_j(|\xi|).
$$
Define $A_j=\Op a_j(x,\xi)$. By formula (\ref{Op(a)})-(\ref{Ker-Conv}), it follows that $A_j$ is represented as a convergent integral operator
$$
A_ju(x)=\int_{\Gp}u(y)\mathcal{K}_j(x,y^{-1}x)dy,
$$
where the kernel
$$
\mathcal{K}_j(x,y)=\left(a_j(x,\cdot)\right)^\vee(y)
=\sum_{\xi\in\DuGp}d_\xi\Tr\left(a_j(x,\xi)\cdot\xi(y)\right).
$$
We fix $Q$ to be the fundamental tuple, as defined in (\ref{QFund}). For any multi-index $\beta$, by definition of RT difference operator,
$$
Q^\beta(y)\mathcal{K}_j(x,y)
=\sum_{\xi\in\DuGp}d_\xi\Tr\left((\Df_{Q^\beta,\xi}a_j)(x,\xi)\cdot\xi(y)\right).
$$
Using the Cauchy-Schwartz inequality for trace inner product (note that if $A$ is an $n\times n$ matrix then $\llbracket A\rrbracket^2\leq n\|A\|^2$), Weyl-type Lemma \ref{Weyl}, the Leibniz type property (\ref{Leibniz}) and Theorem \ref{Multm}, we obtain
\begin{equation}\label{SteinIneq1}
\begin{aligned}
\left|Q^\beta(y)\mathcal{K}_j(x,y)\right|
&\lesssim
\sum_{\xi\in \mathfrak{S}[0,2^{j+1}]}d_\xi^2\cdot\mathbf{M}^m_{0,|\beta|;Q}(a)\size[\xi]^{m-|\beta|}\\
&\lesssim 2^{jn}\cdot\mathbf{M}^m_{0,|\beta|;Q}(a)2^{j(m-|\beta|)}.
\end{aligned}
\end{equation}
Summing over $\beta$ whose whose length is $\leq L$ for $L\in\mathbb{N}$, using that the only common zero of the $Q_i$'s is the identity element, then Fischer's lemma \ref{Fisqq'Lem} for any other strongly RT-admissible tuple $q=(q_i)_{i=1}^M$, we find
\begin{equation}\label{SteinIneq2}
\begin{aligned}
\big(1+|2^j\dist(y,e)|^2\big)^{L}\cdot|\mathcal{K}_j(x,y)|^2
&\lesssim\left[\mathbf{M}^m_{0,L;Q}(a)\right]^22^{2jm+2jn}\\
&\simeq_q\left[\mathbf{M}^m_{0,L;q}(a)\right]^22^{2jm+2jn}.
\end{aligned}
\end{equation}
Hence we can estimate the weighted $L^2$ norm of the kernel:
$$
\begin{aligned}
\int_{\Gp}\big(1+|2^j\dist(x,y)|^2&\big)^{[n/2]+1}|\mathcal{K}_j(x,y^{-1}x)|^2dy\\
&\leq \int_{\Gp}\frac{\big(1+|2^j\dist(x,y)|^2\big)^{n+2}}{\big(1+|2^j\dist(x,y)|^2\big)^{n+1-[n/2]}}\cdot|\mathcal{K}_j(x,y^{-1}x)|^2dy\\
&\lesssim\left[\mathbf{M}^m_{0,n+2;q}(a)\right]^22^{2jm+jn}
\cdot 2^{jn}\int_{\Gp}\frac{dy}{\big(1+|2^j\dist(x,y)|^2\big)^{n+1-[n/2]}}\\
&\lesssim\left[\mathbf{M}^m_{0,n+2;q}(a)\right]^22^{2jm+jn}.
\end{aligned}
$$
Here we estimate the last integral by splitting it into two parts $\dist(x,y)<R$ and $\dist(x,y)\geq R$, where $R$ is the injective radius of $\Gp$; the first part can then be estimated as an Euclidean integral, while the second is just bounded by $2^{-nj}$, and they both absorbs the factor $2^{jn}$.

We then proceed to $X^\alpha\circ A_j$, whose integral kernel will be denoted as $\mathcal{K}_{j,\alpha}(x,y)$, which is just $\mathcal{K}_j(x,y)$ when $\alpha=0$. Using the quantization formulas (\ref{Op(a)})-(\ref{Op-Symbol}), we find that, with $\sigma_{X_i}(x,\xi)=\xi^*(x)\cdot(X_i\xi)(x)$ being the symbol of $X_i$ (which is of class $S^1_{1,0}(\Gp)$ by Fischer's theorem), the symbol of $X_i\circ A_j$ is 
$$
(X_{i;x}a_j)(x,\xi)+\sigma_{X_i}(x,\xi)\cdot a_j(x,\xi)
\in \mathscr{S}^{m+1}_{1,1}(\Gp).
$$
By induction on $|\alpha|$, we obtain 
$$
\int_{\Gp}\big(1+|2^j\dist(x,y)|^2\big)^{[n/2]+1}|\mathcal{K}_{j,\alpha}(x,y^{-1}x)|^2dy
\lesssim\left[\mathbf{M}^m_{|\alpha|,n+2;q}(a)\right]^22^{2j(m+|\alpha|)+jn}.
$$
Applying the Cauchy-Schwarz inequality to $\int \mathcal{K}_{j,\alpha}(x,y^{-1}x)u(y)dy$, we find that for any $u\in L^2$,
$$
\begin{aligned}
\int_{\Gp}|(X^\alpha \circ A_j)u(x)|^2dx
&\lesssim\left[\mathbf{M}^m_{|\alpha|,n+2;q}(a)\right]^22^{2j(m+|\alpha|)+jn}
\int_{\Gp}\int_{\Gp}\frac{|u(y)|^2}{\big(1+|2^j\dist(x,y)|^2\big)^{[n/2]+1}}dydx\\
&=C\left[\mathbf{M}^m_{|\alpha|,n+2;q}(a)\right]^22^{2j(m+|\alpha|)}\|u\|_{L^2}^2,
\end{aligned}
$$
or 
\begin{equation}\label{Aj}
\|X^\alpha\circ A_ju\|_{L^2}\lesssim_\alpha \mathbf{M}^m_{|\alpha|,n+2;q}(a) 2^{j(m+|\alpha|)}\|u\|_{L^2}.
\end{equation}

Now substitute $u$ by any Littlewood-Paley building block $\vartheta_k\big(|\nabla|\big) f$ of $f\in \mathcal{D}(\Gp)$ in (\ref{Aj}). Since the support of $a_j(x,\xi)$ with respect to $\xi$ is contained in $\mathfrak{S}[0,2^{j+1}]$, we have
$$
A_jf=\sum_{k:|k-j|\leq3}A_j\vartheta_k\big(|\nabla|\big)f.
$$
Applying (\ref{Aj}) with $u=\vartheta_k\big(|\nabla|\big) f$, we have
$$
\begin{aligned}
\|X^\alpha A_jf\|_{L^2}
&\lesssim_\alpha \mathbf{M}^m_{|\alpha|,n+2;q}(a) 2^{j(m+|\alpha|)}\sum_{k:|k-j|\leq3}\big\|\vartheta_k\big(|\nabla|\big) f\big\|_{L^2}\\
&\lesssim_\alpha \mathbf{M}^m_{|\alpha|,n+2;q}(a) 2^{j(|\alpha|-s)}\sum_{k:|k-j|\leq3}2^{k(s+m)}\big\|\vartheta_k\big(|\nabla|\big) f\big\|_{L^2}.
\end{aligned}
$$
Since $\sum_{k=1}^\infty2^{2k(s+m)}\big\|\vartheta_k\big(|\nabla|\big) f\big\|_{L^2}^2\simeq\|f\|_{H^{s+m}}^2$, we apply Lemma \ref{H^ss>0} to obtain
$$
\|\Op(a)f\|_{H^s}
=\left\|\sum_{j=0}^\infty A_jf\right\|_{H^s}
\leq C \mathbf{M}^m_{[s]+1,n+2;q}(a)\|f\|_{H^{s+m}}.
$$
\end{proof}
\begin{remark}
We point out that the proof generalizes without much modification to Zygmund spaces $C^r_*$ with $r>0$. Since this is not needed for our goal, we omit the details.
\end{remark}

We do not know whether the class $\Op\mathscr{S}^m_{1,1}$ given in Definition \ref{S^mrd}-\ref{S^mrdAdmi} coincides with any locally defined H\"{o}rmander class of operators or not. However, this is not a question of concern within our scope, since all we need is a calculus for a suitable subclass of symbols and operators, broad enough to cover operators arising from classical differential operators. 

\subsection{Spectral Condition}
The constructions and theorems so far rely mostly, if not completely, on the compact manifold structure of $\Gp$. From now on, we will exploit more algebraic properties of $\Gp$. In particular, the spectral localization property, namely Corollary \ref{SpecPrd}, will play an important role. Just as in the Euclidean case, we introduce a subclass of $\mathscr{S}^m_{1,1}(\Gp)$ that satisfies a spectral condition. The exposition in this subsection is parallel to that of \cite{Met2008}, Chapter 4, or \cite{Hormander1997}, Chapter 9-10. 

Before we pass to the definition of spectral condition, we need a lemma concerning difference operators acting on symbols with finite Fourier support. Being trivial for Euclidean or toric case, for non-commutative groups it relies on spectral localization.
\begin{lemma}\label{DiffVanish}
Suppose the symbol $a=a(\xi)$ has Fourier support contained in $\mathfrak{S}[\gamma_1,\gamma_2]$. Then with $Q$ being the fundamental tuple defined in (\ref{QFund}), there is a constant $C$ depending only on the algebraic structure of $\Gp$, such that for any $Q_i$, the difference $\Df_{Q_i}a$ vanishes for $|\xi|\notin[\gamma_1-C,\gamma_2+C]$.
\end{lemma}
\begin{proof}
We use formula (\ref{Diff_qa}):
$$
(\Df_{Q_i}a)(\xi)=\int_\Gp Q_i(y)\left(\sum_{\eta\in \mathfrak{S}[\gamma_1,\gamma_2]} d_\eta\Tr\left(a(\eta)\cdot\eta(y)\right)\right)\xi^*(y)dy.
$$
By the proof of Corollary \ref{SpecPrd}, the product $Q_i(y)\eta(y)$ is the sum of matrix entries of irreducible representations whose highest weights $\bm{j}\in\mathscr{L}_+$ satisfy
$$
\bm{J}(\eta)-\bm{\varpi}_i\preceq \bm{j}\preceq \bm{J}(\eta)+\bm{\varpi}_i.
$$
Such $\bm{j}$'s must fall into a bounded parallelogram in $\mathfrak{t}^*$ that contains $\bm{J}(\eta)$, so there must hold 
$$
|\bm{J}(\eta)|-C\leq|\bm{j}|\leq |\bm{J}(\eta)|+C,
$$
where $C\simeq\sum_{i=1}^\varrho|\bm{\varpi}_i|$. 
Hence the Fourier support of 
$$
Q_i(y)\left(\sum_{\eta\in \mathfrak{S}[\gamma_1,\gamma_2]} d_\eta\Tr\left(a(\eta)\cdot\eta(y)\right)\right)
$$
is contained in $\mathfrak{S}\big[\gamma_1-C,\gamma_2+C\big]$, with $C\simeq\sum_{i=1}^\varrho|\bm{\varpi}_i|$.
\end{proof}

We turn to the definition of spectral condition. Since no natural ``addition" is available on $\DuGp$, we restrict ourselves to a narrower definition of spectral condition:
\begin{definition}
Fix $\delta>0$. The subclass $\Sigma_\delta^m(\Gp)\subset \mathscr{S}^m_{1,1}(\Gp)$ consists of all symbols $a\in \mathscr{S}^m_{1,1}(\Gp)$ such that the partial Fourier transform of matrix entries of $a(x,\xi)$ satisfies
$$
\Ft{a_{ij}}(\eta,\xi)
:=\int_{\Gp}a_{ij}(x,\xi)\eta^*(x)dx=0
\quad\mathrm{if}\quad|\eta|\geq \delta\size[\xi],
\quad i,j=1,\cdots,d_\xi,
$$
when $\size[\xi]$ is large enough. This is called the spectral condition with parameter $\delta$.
\end{definition}

With a little abuse of notation, we may write 
$$
\Ft{a}(\eta,\xi)= \big[\Ft{a_{ij}}(\eta,\xi)\big]_{i,j=1}^{d_\xi}\in\mathrm{End}(\Hh[\eta]\otimes\Hh[\xi]),
$$
and abbreviate the spectral condition as $\Ft{a}(\eta,\xi)=0$ for $|\eta|\geq \delta\size[\xi]$. Since the symbol $a$ is smooth in $x$, the spectral condition is only important for high modes, i.e. when $\size[\xi]$ is large: if $a(x,\xi)$ vanishes for all but finitely many $\xi\in\DuGp$, then in fact $a\in \mathscr{S}^{-\infty}$, so the ``low frequency part of $a(x,\xi)$" does not cause any problem in regularity. This is why the spectral condition is required only for \emph{sufficiently large $\size[\xi]$.}

We list some basic properties of $\Sigma^m_\delta(\Gp)$ that will be used repeatedly.

\begin{lemma}\label{Sigmadelta}
(1) If $a\in\Sigma^m_\delta(\Gp)$, $b\in\Sigma^{m'}_{\delta'}(\Gp)$, then $a\cdot b\in\Sigma^{m+m'}_{\delta+\delta'}(\Gp)$.

(2) If $a\in\Sigma^m_\delta(\Gp)$, $X\in\mathfrak{g}$ is a left-invariant vector field, then the symbol of $X\circ\Op(a)$ is of order $m+1$, and still satisfies a spectral condition with parameter $\delta$.

(3) Suppose $a\in\Sigma^m_\delta$. Let $Q$ be the fundamental tuple as defined in (\ref{QFund}). Then given any multi-index $\alpha$, there is a constant $C_\alpha$ depending only on $\alpha$ and the algebraic structure of $\Gp$, such that the symbol $\Df_Q^\alpha a\in \mathscr{S}^{m-|\alpha|}_{1,1}$, and satisfies the spectral condition with parameter $C_\alpha\delta$.
\end{lemma}
\begin{proof}
Part (1) is a direct consequence of the spectral localization property, namely Corollary \ref{SpecPrd}. For part (2), write $A=\Op(a)$. For a left-invariant vector field $X\in\mathfrak{g}$, the symbol of $X\circ A$ reads
\begin{equation}\label{SymbXaj}
(X_{x}a)(x,\xi)+\sigma_X(x,\xi)\cdot a(x,\xi)\in \mathscr{S}^{m+1}_{1,1},
\end{equation}
as seen in the proof of Stein's theorem. By the right convolution kernel formula (\ref{Symbol-Ker})-(\ref{Ker-Conv}), we find that the right convolution kernel $\mathcal{K}(x,y)$ of $X$ is invariant under left translation with respect to $x$, so in fact $\mathcal{K}(x,y)\equiv \mathcal{K}(e,y)$, and the symbol $\sigma_X(x,\xi)$ is thus independent from $x$. As a result, it is legitimate to write it as $\sigma_X(\xi)\in\mathrm{End}(\Hh[\xi])$, and refer it as a \emph{non-commutative Fourier multiplier}. Consequently, the Fourier transform of (\ref{SymbXaj}) with respect to the $x$ variable evaluated at $\eta\in\DuGp$ is 
$$
\sigma_X(\eta)\cdot\Ft{a}(\eta,\xi)+\left(\sigma_X(\xi)\cdot a\right)^{\wedge}(\eta,\xi),
$$
which still vanishes for $|\eta|\geq\delta\size[\xi]$.

As for part (3), we just have to prove the claim for $|\alpha|=1$ as the rest follows by induction. Theorem \ref{Fischer} ensures that $\Df_Q a\in \mathscr{S}^{m-1}_{1,1}$. Here with a little abuse of notation we use $Q$ to denote any component of the fundamental tuple. To verify the spectral condition, we employ formula (\ref{Diff_qa}):
$$
(\Df_Qa)(x,\xi)=\int_\Gp Q(y)\left(\sum_{\zeta\in\DuGp} d_\zeta\sum_{i,j=1}^{d_\zeta}a_{ij}(x,\zeta)\zeta_{ij}(y)\right)\xi^*(y)dy.
$$
Taking Fourier transform with respect to $x$, evaluating at $\eta\in\DuGp$ and using the spectral condition, we obtain
\begin{equation}\label{Lem3.5Temp}
\Ft{\Df_Qa}(\eta,\xi)
=\sum_{\zeta\in\DuGp:\size[\zeta]\geq \delta^{-1}|\eta|}d_\zeta \sum_{i,j=1}^{d_\zeta}\Ft{a_{ij}}(\eta,\zeta)\otimes
\int_\Gp Q(y)\zeta_{ij}(y)\xi^*(y)dy
\in\mathrm{End}(\Hh[\eta]\otimes\Hh[\xi]).
\end{equation}
From the proof of Lemma \ref{DiffVanish}, we find that $Q(y)\xi^*(y)$ has Fourier support contained in $\mathfrak{S}\big[|\xi|-C,|\xi|+C\big]$, where $C$ depends only on the algebraic property of $\Gp$.

Now suppose $|\eta|\geq 2\delta\size[\xi]$. Then for each summand in (\ref{Lem3.5Temp}), the representation $\zeta$ satisfies $|\zeta|\geq 2\size[\xi]$. By Schur orthogonality, matrix entries of $\zeta$ is thus $L^2$ orthogonal to matrix entries of $\bar Q(y)\xi(y)$. This shows that for $|\eta|\geq 2\delta\size[\xi]$, each summand in (\ref{Lem3.5Temp}) vanishes. 
\end{proof}

We are now at the place to state the fundamental theorem for para-differential calculus: (1,1) symbols satisfying a suitable spectral condition give rise to operators bounded in all Sobolev spaces. 

\begin{theorem}\label{Stein'}
Let $\delta\in(0,1/2)$, $m\in\mathbb{R}$. Suppose $a\in \Sigma_\delta^m(\Gp)$. Then the operator $\Op(a)$ maps $H^{s+m}(\Gp)$ to $H^{s}(\Gp)$ continuously for all $s\in\mathbb{R}$. Quantitatively, with $s_+=0$ for $s\leq0$ and $s_+=s$ for $s>0$, for any strongly RT-admissible tuple $q$,
$$
\|\Op(a)f\|_{H^s}\leq C_{s,\delta;q}\mathbf{M}^m_{[s_+]+1,n+2;q}(a)\|f\|_{H^{s+m}}.
$$
\end{theorem}

Before we proceed to the proof, we still need a summation Lemma in Sobolev spaces.

\begin{lemma}\label{H^ssReal}
Let $s\in\mathbb{R}$, $l>s$ be a positive integer, and $\gamma>0$. Fix a basis $X_1,\cdots,X_n$ of $\mathfrak{g}$, and define $X^\alpha$ as in Proposition \ref{NormalOrder}. There is a constant $C=C(s,l,\gamma)$ with the following properties. If $\{f_k\}_{k\geq0}$ is a sequence in $H^l(\Gp)$, such that for any multi-index $\alpha$ with $|\alpha|\leq l$, there holds
$$
\|X^\alpha f_k\|_{L^2}\leq 2^{k(|\alpha|-s)}c_k,\quad (c_k)\in \ell^2(\mathbb{N}),
$$
and furthermore $\Ft{f}_k$ is supported in $\{\xi:\size[\xi]\geq\gamma 2^k\}$, then $\sum_k f_k\in H^s(\Gp)$ and 
$$
\left\|\sum_{k\geq0}f_k\right\|_{H^s}^2
\leq C\sum_{k\geq0}c_k^2.
$$
\end{lemma}
\begin{proof}
The proof is quite similar for that of Lemma \ref{H^ss>0}. The only difference is as follows: the spectral condition $\mathrm{supp}\Ft{f}_k\subset\{\xi:\size[\xi]\geq\gamma 2^k\}$ implies that for some integer $N\sim\log\gamma$, the Littlewood-Paley building block $\vartheta_j\big(|\nabla|\big)f_k$ vanishes for $j<k-N$. So for $f=\sum_kf_k$, we obtain similarly as in Lemma \ref{H^ss>0},
$$
2^{js}\|\vartheta_j\big(|\nabla|\big) f\|_{L^2}
\leq C\sum_{k:k<j}2^{(l-s)(k-j)}c_k
+\sum_{k:j\leq k<j+N}2^{(j-k)s}c_k.
$$
The first sum gives rise to terms of an $\ell^2$ sequence again by Lemma \ref{l2Seq}, while the second obviously gives rise to terms of an $\ell^2$ sequence. Thus again $\{2^{js}\big\|\vartheta_j\big(|\nabla|\big) f\big\|_{L^2}\}_{j\geq0}$ is an $\ell^2$ sequence. The rest of the proof is identical to that of Lemma \ref{H^ss>0}.
\end{proof}

\begin{proof}[Proof of Theorem \ref{Stein'}]
The proof is still very much similar to that of Theorem 4.3.4. in \cite{Met2008}, and proceeds basically parallelly as Stein's theorem. We still decompose $\Op(a)f$ into a dyadic sum $\sum_{j\geq0}A_jf$, where $A_j$ is the quantization of $a_j(x,\xi):=a(x,\xi)\vartheta_j(|\xi|)$. The estimate of operator norms for $X^\alpha\circ A_j$ is identical to that in the proof of Theorem \ref{SteinTheorem}. There are only some minor differences: since $\Ft{a}(\eta,\xi)=0$ for $|\eta|>\delta\size[\xi]$, by the quantization formula (\ref{Op(a)})-(\ref{Op-Symbol}) we have
$$
\Ft{A_jf}(\eta)=\int_{\Gp}\sum_{\xi\in \mathfrak{S}[2^{j-1},2^j]}d_\xi\Tr\left(\xi(x)a_j(x,\xi)\Ft{f}(\xi)\right)\eta^*(x)dx.
$$
By spectral localization property \ref{SpecPrd} and the spectral condition $\Ft{a_j}(\eta,\xi)=0$ for $|\eta|\geq\delta 2^j$ (since $|\xi|\leq 2^j$), we find that the integrand $\Tr\left(\xi(x)a_j(x,\xi)\Ft{f}(\xi)\right)$ has Fourier support (with respect to $x$) contained in 
$$
\left\{\eta\in\DuGp:|\eta|\geq c\left(\frac{1}{2}-\delta\right)2^j\right\},
$$
where $c\in(0,1)$ depends only on the algebraic structure of $\Gp$, i.e. $\Ft{A_jf}(\eta)=0$ for $|\eta|\leq c(1/2-\delta)2^j$. By Lemma \ref{Sigmadelta} and induction on $|\alpha|$, the symbol of $X^\alpha\circ a$ still satisfies the spectral condition with parameter $\delta$. To summarize, for any $f\in\mathcal{D}'(\Gp)$, the Fourier support of $X^\alpha A_jf$ is within $\{\eta\in\DuGp:|\eta|>c(1/2-\delta)2^j\}$. 

With exactly the same argument as in the proof of Theorem \ref{SteinTheorem}, the estimate
$$
\|X^\alpha A_jf\|
\lesssim_\alpha \mathbf{M}^m_{|\alpha|,n+2;q}(a) 2^{j(|\alpha|-s)}\sum_{k:|k-j|\leq3}2^{k(s+m)}\big\|\vartheta_k\big(|\nabla|\big) f\big\|_{L^2}
$$
is valid. We can then apply Lemma \ref{H^ssReal} to conclude.
\end{proof}
\begin{remark}\label{HorRegSym}
The proof for Theorem \ref{SteinTheorem} and \ref{Stein'} proceed with exactly the same idea as in the Euclidean case, since the product spectral localization property does not differ very much from the Euclidean case under dyadic decomposition. We thus expect that H\"{o}rmander's characterization of the class $\tilde\Psi^m_{1,1}(\mathbb{R}^n)$, namely (1,1) pseudo-differential operators that are continuous on all Sobolev spaces (see \cite{Hormander1988} or Chapter 9 of \cite{Hormander1997} for the details), admits a direct generalization to $\Gp$. But we content ourselves with operators satisfying the spectral condition, because this is what we actually need for para-differential calculus. We only write $\tilde{\mathscr{S}}^{m}_{1,1}(\Gp)$ for the subclass of $(1,1)$ symbols whose quantization are continuous on every Sobolev space, and do not characterize it explicitly. The inclusion $\Sigma_{<1/2}^m\subset\tilde{\mathscr{S}}^{m}_{1,1}$ is quite sufficient for most of our applications.
\end{remark}

\subsection{Para-product and Para-linearization on \texorpdfstring{$\Gp$}{a}}
As an initial application, we are able to recover the para-product estimate and Bony's para-linearization theorem on $\Gp$, just as in \cite{BGdP2021}, but this time with a global notion of para-product. We point out that in \cite{KR2006}, the authors also constructed para-product in terms of spectral calculus on a general compact surface. Since $\Gp$ obviously enjoys more structure than a general manifold, we are able to recover almost everything in the Euclidean setting.

Given two distributions $a$ and $u$, we define the \emph{para-product} $T_au$ by
\begin{equation}\label{T_au}
T_au
:=\int_1^\infty\left[\phi_{2^{-10}t}\big(|\nabla|\big)a\right]\cdot\psi_t\big(|\nabla|\big)u\frac{dt}{t},
\end{equation}
where $\phi$ and $\psi$ are as in the continuous Littlewood-Paley decomposition (\ref{LPCont}). Similarly as in the Euclidean case, the gap 10 is inessential. By the spectral localization property i.e. Corollary \ref{SpecPrd}, the integrand has Fourier support contained in $\mathfrak{S}[ct,c^{-1}t]$ for some constant $c\in(0,1)$, depending only on the algebraic structure of $\Gp$. The symbol of $T_a$ is 
$$
\sigma[T_a](x,\xi)
=\int_1^\infty\left[\phi_{2^{-10}t}\big(|\nabla|\big)a(x)\right]\cdot\psi_t\big(|\xi|\big)\cdot\I[\xi]\frac{dt}{t}.
$$
\begin{proposition}
If $a\in L^\infty$, the symbol $\sigma[T_a]$ of the para-product operator $T_a$ is of class $\Sigma^0_{2^{-10}}$.
\end{proposition}
\begin{proof}
The spectral condition is obvious since $\psi_t(|\xi|)\neq0$ only for $|\xi|\leq t$. We turn to estimate the operator norm of $X^\alpha_x\Df_{q,\xi}^\beta \sigma[T_a](x,\xi)(x,\xi)$ on $\Hh[\xi]$, for a given basis $X_1,\cdots,X_n$ of $\mathfrak{g}$ and $q$ is a strongly RT-admissible tuple. But this is immediate from Lemma \ref{CutOffKer} and Theorem \ref{Multm}, and quantatively $\mathbf{M}^0_{k,l;q}\big(\sigma[T_a]\big)\lesssim_{k,l}\size[\xi]^{k-l}|a|_{L^\infty}$.
\end{proof}

By Theorem \ref{Stein'}, the para-product operator $T_a$ is thus continuous from $H^s$ to itself for all $s\in\mathbb{R}$. This of course also admits a more direct proof. We can simply insert the multiplier corresponding to a function $h\in C_0^\infty(0,\infty)$, such that $h(\lambda)\equiv1$ for $\lambda\in[c,c^{-1}]$ and $h\geq0$. Then the Fourier support of $T_au$ is away from the trivial representation $\xi\equiv1$, and
$$
T_au
=\int_0^\infty h_t\big(|\nabla|\big)\Big(\phi_{2^{-10}t}\big(|\nabla|\big)a\cdot\psi_t\big(|\nabla|\big)u\Big)\frac{dt}{t}.
$$
Write $v(t)=\phi_{2^{-10}t}\big(|\nabla|\big)a\cdot\psi_t\big(|\nabla|\big)u$ for the integrand. It has Fourier support contained in $\mathfrak{S}[ct,c^{-1}t]$. By Bernstein's inequality and Lemma \ref{CutOffKer}, we find that
$$
\begin{aligned}
\|v(t)\|_{H^s}
&\lesssim t^{s}\left|\phi_{2^{-10}t}\big(|\nabla|\big)a\right|_{L^\infty}\big\|\psi_t\big(|\nabla|\big)u\big\|_{L^2}\\
&\lesssim t^{s}|a|_{L^\infty}\big\|\psi_t\big(|\nabla|\big)u\big\|_{L^2}.
\end{aligned}
$$
So by the Peter-Weyl theorem,
$$
\begin{aligned}
\|T_au\|_{H^s}^2
&\simeq\sum_{\xi\in\DuGp}d_\xi|\xi|^{2s}\left\llbracket\int_0^\infty h_t(|\xi|)\cdot
\Ft{v(t)}(\xi)\frac{dt}{t}\right\rrbracket^2\\
&\leq \sum_{\xi\in\DuGp}d_\xi\int_0^\infty
\big|h_t(|\xi|)\big|^2\frac{dt}{t}\cdot\int_0^\infty
|\xi|^{2s}\big\llbracket\Ft{v(t)}(\xi)\big\rrbracket^2\frac{dt}{t}\\
&\lesssim \int_0^\infty\|v(t)\|_{H^s}^2\frac{dt}{t}
\lesssim |a|_{L^\infty}^2\|u\|_{H^s}^2,
\end{aligned}
$$
where we used Proposition \ref{LPHs} at the last step. Note that the integrals in the above are not singular since the integrands all vanish near $t=0$.

Just as in the Euclidean case, we have the para-product decomposition on $\Gp$:
\begin{theorem}\label{ParaPrd}
If $s>0$, $a,u\in (L^\infty\cap H^s)(\Gp)$, then
$$
au=T_au+T_ua+R(a,u),
$$
where the smoothing remainder 
$$
\|R(a,u)\|_{H^s}\lesssim |a|_{L^\infty}\|u\|_{H^s}.
$$
\end{theorem}
\begin{proof}
It suffices to investigate the remainder $R(a,u)$. Using the continuous Littlewood-Paley decomposition (\ref{LPCont}),
$$
R(a,u)=\int_{\substack{2^{-10}\leq t_1/t_2\leq 2^{10}\\ 1\leq t_1,t_2}}\psi_{t_1}\big(|\nabla|\big)a\cdot\psi_{t_2}\big(|\nabla|\big)u\cdot \frac{dt_1dt_2}{t_1t_2}
+\phi\big(|\nabla|\big)a\cdot\phi\big(|\nabla|\big)u.
$$
Only the integral is non-trivial. We make a change of variable $(t_1,t_2)=t(\cos\varphi,\sin\varphi)$, so that it equals
$$
\int_{2^{-10}\leq \tan\varphi\leq 2^{10}}\left(\int_1^\infty\psi_{t\cos\varphi}\big(|\nabla|\big)a\cdot\psi_{t\sin\varphi}\big(|\nabla|\big)u\cdot \frac{dt}{t}\right)\sin\varphi d\varphi.
$$
Then $R(a,u)=\Op(R[a])u$, where the symbol
$$
R[a](x,\xi)=
\int_{2^{-10}\leq \tan\varphi\leq 2^{10}}\left(\int_1^\infty\psi_{t\cos\varphi}\big(|\nabla|\big)a
\cdot\psi_{t\sin\varphi}(|\xi|)\I[\xi]\cdot \frac{dt}{t}\right)\sin\varphi d\varphi.
$$
Just as in the proof of the previous proposition, we employ Lemma \ref{CutOffKer} and Theorem \ref{Multm} to conclude, for a strongly RT-admissible tuple $q$,
$$
\mathbf{M}^0_{k,l;q}(R[a])\lesssim_{k,l}\size[\xi]^{k-l}|a|_{L^\infty}.
$$
Thus by Stein's theorem, $\Op(R[a])$ maps $H^s$ to itself for $s>0$.
\end{proof}
\begin{remark}\label{R(a)Symbol}
If in addition $a\in C^r_*(\Gp)$, then from the spectral characterization of Zygmund spaces, we find that in fact $R[a]\in \mathscr{S}^{-r}_{1,1}$, so 
$$
\|R(a,u)\|_{H^s}\lesssim|a|_{C^r_*}\|u\|_{H^{s-r}}.
$$
\end{remark}

We also have a direct generalization of Bony's para-linearization theorem, proposed by Bony \cite{Bony1981}. Such a result is neither proved true or false for general compact manifolds in \cite{KR2006}. 
\begin{theorem}[Bony]\label{Bony}
Suppose $F\in C^\infty(\mathbb{C};\mathbb{C})$ (understood as smooth mapping on the plane instead of holomorphic function), $F(0)=0$. Let $r>0$, and suppose $u\in C^r_*$. Write $u_t=\phi(D/t)u$, where $\phi$ is the function in the Littlewood-Paley decomposition (\ref{LPCont}). Then with the symbol
$$
l_u(x,\xi):=\int_1^\infty F'(u_t(x))\cdot\psi_t\big(|\xi|\big)\cdot\I[\xi]\frac{dt}{t},
$$
there holds $F(u)=F(u_1)+\Op(l_u)u$, and we have the para-linearization formula
$$
F(u)=F(u_1)+T_{F'(u)}u+\Op(R[u])u
$$
with the symbol $R[u](x,\xi)\in \mathscr{S}^{-r}_{1,1}$. Consequently, if $u\in C^r_*\cap H^s$, then $F(u)-T_{F'(u)}u\in H^{s+r}$. In particular, for $s>n/2$, $F(u)-T_{F'(u)}u\in H^{2s-n/2}$.
\end{theorem}
\begin{proof}
The proof resembles that of Theorem 10.3.1. in \cite{Hormander1997} very much. We first set $u_t=\phi_t\big(|\nabla|\big)u$, then define a symbol
$$
l_u(x,\xi):=\int_1^\infty F'(u_t(x))\cdot\psi_t(|\xi|)\I[\xi]\frac{dt}{t}
$$
That $F(u)=F(u_1)+l_u(x,D)u$ is obtained by a direct computation already done when estimating the symbol of $T_a$, which also shows $l_u(x,\xi)\in \mathscr{S}^0_{1,1}$. Working in local coordinate charts, we find $F(u),F'(u)\in C^r_*$. 

Let us set $p(x,\xi)$ as the symbol of the para-product operator $T_{F'(u)}$. For any strongly RT-admissible tuple $q$, we estimate the $\mathbf{M}_{k,l;q}^{-r}$ norm of
$$
R[u](x,\xi):=l_u(x,\xi)-p(x,\xi)
=\int_1^\infty
\Big[F'(u_t(x))-\phi_{2^{-10}t}\big(|\nabla|\big)F'(u(x))\Big]\psi_t(|\xi|)\I[\xi]\frac{dt}{t}.
$$
In view of Lemma \ref{CutOffKer} and Corollary \ref{Fischer} once again, all we need to show is that the magnitude of
\begin{equation}\label{2Alpha}
X^\alpha_x\Big[F'(u_t)-\phi_{2^{-10}t}\big(|\nabla|\big)F'(u)\Big]
\end{equation}
has an upper bound $K_\alpha(|u|_{C^r_*}) t^{|\alpha|-r}$, where $K_\alpha$ is an increasing function. 

For $\alpha=0$, we find the quantity (\ref{2Alpha}) equals
$$
F'(u_t)-\phi_{2^{-10}t}\big(|\nabla|\big)F'(u)
=F'(u_t)-F'(u)+F'(u)-\phi_{2^{-10}t}\big(|\nabla|\big)F'(u).
$$
By Corollary \ref{LPZygCor}, we have $|u-u_t|\lesssim t^{-r}|u|_{C^r_*}$, and since $F'(u)\in C^r_*$, we find that (\ref{2Alpha}) is controlled by $t^{-r}$ for $\alpha=0$. 

For $|\alpha|>r$, the second term in (\ref{2Alpha}) is controlled by $t^{|\alpha|-r}|u|_{C^r_*}$ with the aid of Corollary \ref{LPZygCor}, since $F'(u)\in C^r_*$. We note that being $C^r_*$ is a local property, so it is legitimate to work in any local coordinate chart $\{x\}$ and estimate the usual local partial derivatives of the ambient function. Using the chain rule, we find that the first term in (\ref{2Alpha}) is a $C^\infty$-linear combination of terms of the form
$$
F^{(1+|\alpha|-|\beta|)}(u_t)\left(\partial_x^{\beta_1}u_t\right)^{j_1}\cdots\left(\partial_x^{\beta_k}u_t\right)^{j_k},
$$
where $\beta=\beta_1+\cdots+\beta_k$, and $j_1|\beta_1|+\cdots+j_k|\beta_k|=|\alpha|$. Using Corollary \ref{LPZygCor} again, we can bound this term by
$$
K(|u|_{C^r_*})\prod_{i:\beta_i> r}\left(|u|_{C^r_*}t\right)^{j_i(|\beta_i|-r)}
\cdot\prod_{i:\beta_i\leq r}\left(|u|_{C^r_*}\log(1+t)\right)^{j_i}
=O(t^{|\alpha|-r}).
$$
As for $0<|\alpha|\leq r$, we can just use the classical interpolation inequality for smooth functions with all derivatives bounded, in between 0 and $r+1$:
$$
|D^jf|_{L^\infty}\lesssim_{j,k}|f|_{C^k}^{j/k}\cdot|f|_{L^\infty}^{1-j/k},
\quad
0<j<k.
$$
This implies that (\ref{2Alpha}) is controlled by $t^{|\alpha|-r}$, hence ensures $R[u]\in \mathscr{S}^{-r}_{1,1}$. The claim of the theorem now follows from Stein's theorem.
\end{proof}

\section{Para-differential Operators on Compact Lie Group}\label{4}
\subsection{Rough Symbols and Their Smoothing}
Because of the spectral localization property on $\Gp$, the theory of para-differential operators resembles that on $\mathbb{R}^n$ very much. The details of the latter may be found in Chapter 9-10 in \cite{Hormander1997}. For completeness and convenience of future reference, we still write down the details here.

\begin{definition}\label{Rough}
For $r\geq0$ and $m\in\mathbb{R}$, define the symbol class $\mathcal{A}_r^m(\Gp)$ to be the collection of all symbols $a(x,\xi)$ on $\Gp\times\DuGp$, such that if $q=(q_i)_{i=1}^M$ is a strongly RT-admissible tuple, then
$$
\big\|\Df_{q,\xi}^\beta a(x,\xi)\big\|_{r;x}\leq C_\beta\size[\xi]^{m-|\beta|},\quad\forall\xi\in\DuGp.
$$
Here the norm $\|a(x,\xi)\|_{r;x}$ for $a(\cdot,\xi):\Gp\to\mathrm{End}(\Hh[\xi])$ is defined by
$$
\|a(x,\xi)\|_{r;x}:=
\sup_{B^*\in\mathrm{End}(\Hh[\xi])^*}\frac{\big|\langle B^*, a(x,\xi)\rangle\big|_{C^r_{*x}}}{\|B^*\|},
$$
as in Remark \ref{VectZyg}, that is we consider $\mathrm{End}(\Hh[\xi])$ merely as a normed space instead of a normed algebra. Introduce the following norm on $\mathcal{A}_r^m(\Gp)$:
$$
\mathbf{W}^{m;r}_{l;q}(a):=\sup_{\xi\in\DuGp}\sum_{|\beta|\leq l}\size[\xi]^{|\beta|-m}\big\|\Df_{q,\xi}^\beta a(x,\xi)\big\|_{r;x}
$$
\end{definition}
We note that Fischer's Lemma \ref{Fisqq'Lem} ensures that the definition of $\mathcal{A}^m_r(\Gp)$ does not depend on the choice of $q$, so we can choose any set of functions $q$ that facilitates our computation. In particular, the fundamental tuple $Q$ defined in (\ref{QFund}) is a convenient choice.

\begin{definition}[Admissible cut-off]\label{AdmCutoff}
An admissible cut-off function $\chi$ with parameter $\delta\in(0,1/2)$ is a smooth function on $\mathbb{R}\times\mathbb{R}$ such that
$$
\chi(\mu,\lambda)=\left\{
\begin{aligned}
1,&\quad |\mu|\leq\frac{\delta}{2}\size[\lambda],\\
0,&\quad |\mu|\geq\delta\size[\lambda],
\end{aligned}
\right.
$$
and there also holds
$$
|\partial_\mu^k\partial_\lambda^l\chi(\mu,\lambda)|
\lesssim_{k,l}\langle\lambda\rangle^{-k-l}.
$$
\end{definition}
An obvious choice of admissible cut-off function is the one that we used to construct para-product (\ref{T_au}):
\begin{equation}\label{AdmCutoffdelta}
\chi(\mu,\lambda)=\int_1^\infty\phi_{\delta t/2}(\mu)\psi_t(\lambda)\frac{dt}{t}.
\end{equation}
However, it will be shown shortly that there is a lot of flexibility in choosing an admissible cut-off function.

\begin{definition}[Para-differential Operator]
Let $\chi$ be an admissible cut-off function with some parameter $\delta$. For $a\in\mathcal{A}_r^m(\Gp)$, set $a^\chi (x,\xi)=\chi(|\nabla_x|,|\xi|)a(x,\xi)$, the regularized symbol corresponding to $a$. Define the para-differential operator $T_a^\chi$ corresponding to $a\in\mathcal{A}_r^m(\Gp)$ as
$$
T_a^\chi u(x):= \Op(a^\chi)u(x).
$$
\end{definition}

We immediately deduce that $a^\chi$ is a (1,1) symbol satisfying the spectral condition with parameter $\delta$ with improved growth estimate:
\begin{proposition}\label{ParaDiff1}
Suppose $r\geq0$. Fix a basis $X_1,\cdots,X_n$ of $\mathfrak{g}$, and define $X^\alpha$ as in Proposition \ref{NormalOrder}. Suppose $a\in\mathcal{A}_r^m(\Gp)$ and $\chi$ is an admissible cut-off function with parameter $\delta$. If $r=0$, then $a^\chi \in\Sigma^{m+\varepsilon}_\delta(\Gp)\subset \mathscr{S}^{m+\varepsilon}_{1,1}(\Gp)$. If $r>0$, then $a^\chi \in\Sigma^{m}_\delta(\Gp)\subset \mathscr{S}^{m}_{1,1}(\Gp)$ for some $\delta>0$. In fact, for any strongly RT-admissible tuple $q$, we have
$$
\big\|X_x^\alpha\Df_{q,\xi}^\beta a^\chi (x,\xi)\big\|
\lesssim_{\alpha,\beta;q}\left\{
\begin{aligned}
    & \mathbf{W}^{m;r}_{|\beta|;q}(a)\cdot\size[\xi]^{m+(|\alpha|-r)_+-|\beta|} &\quad |\alpha|\neq r \\
    & \mathbf{W}^{m;r}_{|\beta|;q}(a)\cdot\size[\xi]^{m-|\beta|}\log(1+\size[\xi]) &\quad |\alpha|=r
\end{aligned}
\right.,
$$
where $s_+=\max(s,0)$. The second inequality matters only when $r$ is a positive integer.
\end{proposition}
\begin{proof}
That $a^\chi$ satisfies the spectral condition with parameter $\delta$ is immediate from the definition: the Fourier cut-off operator $\chi(|\nabla_x|,|\xi|)$ annihilates all frequencies of size greater than $\delta\size[\xi]$. We turn to verify the estimates for the $\mathscr{S}^m_{1,1}$ norms of $a^\chi$. By Fischer's lemma i.e. Lemma \ref{Fisqq'Lem}, it suffices to prove the inequality for $q=Q$, the fundamental tuple (\ref{QFund}). The advantage is that the corresponding RT difference operators have Leibniz type property, and all the $\Df_Q^\beta a^\chi$'s also satisfy spectral conditions (with parameters proportional to $\delta$) by Lemma \ref{Sigmadelta}.

For $\beta=0$, the estimate for $\|X^\alpha_xa^\chi(x,\xi)\|$ follows from the vector-valued version of Corollary \ref{LPZygCor} (see Remark \ref{VectZyg}), where the normed space considered is $\Hh[\xi]$. To obtain the estimate for $\|X_x^\alpha\Df_{Q,\xi}^\beta a^\chi (x,\xi)\|
$ with $\beta\neq0$, we just have to apply Theorem \ref{Multm}.
\end{proof}

We notice that the norm of $a^\chi(x,\xi)$ grows slower upon differentiation in $x$ than generic $\mathscr{S}^m_{1,1}$ symbols. This property is called by H\"{o}rmander as having \emph{reduced order $m-r$} for symbols on $\mathbb{R}^n$. 

The next proposition shows that the para-differential operator $T_a^\chi=\Op(a^\chi)$ extracts the ``highest order differentiation" from $\Op(a)$. The para-product decomposition can be seen as a special case of it.

\begin{proposition}\label{a-a^chi}
Suppose $r>0$, $a\in\mathcal{A}_r^m(\Gp)$. Fix a basis $X_1,\cdots,X_n$ of $\mathfrak{g}$, and define $X^\alpha$ as in Proposition \ref{NormalOrder}. If $\chi$ is an admissible cut-off function, then $a-a^\chi \in\mathcal{A}_0^{m-r}$, and in fact for any strongly RT-admissible tuple $q$,
$$
\sup_{x\in\Gp}\big\|\Df_{q,\xi}^\beta(a- a^\chi)(x,\xi)\big\|
\lesssim_{\beta,q} \mathbf{W}^{m;r}_{|\beta|;q}(a)\cdot\langle\xi\rangle^{m-r-|\beta|}.
$$
\end{proposition}
\begin{proof}
We still just need to prove for $q=Q$, the fundamental tuple. By induction, it suffices to prove for $|\beta|=1$. By the Leibniz type property (\ref{Leibniz}), setting $M_i$ to be the dimension of the $i$'th fundamental representation,
$$
\Df_{Q_i,\xi}(a-a^\chi)(x,\xi)
=\big(\Df_{Q_i,\xi}a-(\Df_{Q_i,\xi}a)^\chi\big)(x,\xi)
-\sum_{j,k=1}^{M_i}c^i_{j,k}\Df_{Q_j,\xi}(\chi\cdot\I[\xi])\big(|\nabla_x|,|\xi|\big)\Df_{Q_k}a(x,\xi).
$$
By Lemma \ref{Sigmadelta}, $\Df_{Q_i}a$ satisfies the spectral condition with parameter proportional to $\delta$, so we may simply apply Corollary \ref{LPZygCor} (and Remark \ref{VectZyg}) to conclude that 
$$
\sup_{x\in\Gp}\left\|\big(\Df_{Q_i,\xi}a-(\Df_{Q_i,\xi}a)^\chi\big)(x,\xi)\right\|
\lesssim_{Q} \mathbf{W}^{m;r}_{1;Q}(a)\cdot\langle\xi\rangle^{m-r-1}.
$$
The sum $\sum_{j,k=1}^{M_i}$ needs greater care to handle. For a fixed $\eta\in\DuGp$, Theorem \ref{Multm} ensures that the symbol $\Df_{Q_j,\xi}(\chi\cdot \I[\xi])\big(|\eta|,|\xi|\big)$ is in a bounded subset in $\mathscr{S}^{m-1}_{1,1}$, and by Lemma \ref{DiffVanish} it does not vanish only if 
$$
\frac{1}{\delta}|\eta|-C\leq|\xi|\leq\frac{2}{\delta}|\eta|+C,
$$
where $C$ depends on the algebraic structure of $\Gp$ only. Consequently, $\Df_{Q_j,\xi}(\chi\cdot \I[\xi])\big(|\eta|,|\xi|\big)\neq0$ only if $|\eta|/\size[\xi]\simeq \delta$, so when the highest weight of $\xi$ is large enough, the Fourier support of 
$$
\Df_{Q_j,\xi}(\chi\cdot\I[\xi])\big(|\nabla_x|,|\xi|\big)\Df_{Q_k}a(x,\xi)
$$
with respect to $x$ is contained in $\mathfrak{S}[C\delta\size[\xi],C'\delta\size[\xi]]$, with $C,C'>0$ depending on the algebraic structure of $\Gp$ only. Corollary \ref{LPZygCor} and Remark \ref{VectZyg} then implies that
$$
\sup_{x\in\Gp}\left\|\Df_{Q_j,\xi}(\chi\cdot\I[\xi])\big(|\nabla_x|,|\xi|\big)\Df_{Q_k}a(x,\xi)\right\|
\lesssim_Q
\mathbf{W}^{m;r}_{1;Q}(a)\cdot\langle\xi\rangle^{m-r-1}.
$$
\end{proof}

We next estimate the difference $a^{\chi_1}-a^{\chi_2}$ for different choices of admissible cut-off functions. The result shows that the choice of $\chi$ is in fact very flexible:
\begin{proposition}\label{FreedomForCut-off}
Suppose $r\geq0$, $\chi_1,\chi_2$ are two admissible cut-off functions with parameters $0<\delta_1<\delta_2<1/2$ respectively. Fix a basis $X_1,\cdots,X_n$ of $\mathfrak{g}$, and define $X^\alpha$ as in Proposition \ref{NormalOrder}. Then for $a\in\mathcal{A}_r^m$, the difference $a^{\chi_1}- a^{\chi_2}\in \Sigma^{m-r}_{\delta_2}$. In fact for any strongly RT-admissible tuple $q$,
$$
\big\|X^\alpha_x\Df_{q,\xi}^\beta (a^{\chi_1}- a^{\chi_2})(x,\xi)\big\|
\lesssim_{\alpha,\beta,q}
\mathbf{W}^{m;r}_{|\beta|;q}(a)\cdot\langle\xi\rangle^{m-r+|\alpha|-|\beta|}.
$$
Thus for any $s\in\mathbb{R}$, $T_a^{\chi_1}-T_a^{\chi_2}$ maps $H^{s+m}$ continuously to $H^s$, and
$$
\|(T_a^{\chi_1}-T_a^{\chi_2})u\|_{H^s}
\lesssim_{s,q}\mathbf{W}^{m;r}_{n+2;q}(a)\|u\|_{H^{s+m-r}}.
$$
\end{proposition}
\begin{proof}
This follows immediately from an observation and Corollary \ref{LPZygCor} (also Remark \ref{VectZyg}): the Fourier support of $\chi_1(\eta,\xi)-\chi_2(\eta,\xi)$ with respect to $x$ is in the set $\mathfrak{S}[\delta_1\size[\xi],\delta_2\size[\xi]]$.
\end{proof}
Thus, given a rough symbol $a\in\mathcal{A}^m_r$ with $r>0$, it is legitimate to abbreviate the dependence on admissible cut-off function, and define \emph{the} para-differential operator $T_a$ modulo $\Op\big(\Sigma^{m-r}_{<1/2}\big)$.

\subsection{Symbolic Calculus of Para-differential Operators}
In this subsection, we state the core results in our toolbox: the composition and adjoint formula for para-differential operators. The spirit of proof does not differ much from that in Métivier's book \cite{Met2008}, but technically we have to be extra careful dealing with the operator norm of symbols. The spectral localization property i.e. Corollary \ref{SpecPrd} and results obtained in Subsection 4.1. will be used repeatedly.

\begin{definition}
Let $r>0$ be a real number. Let $q$ be a RT-admissible tuple, and $X_q^{(\alpha)}$ be the left-invariant differential operators as in Proposition \ref{TaylorGp}. If $a,b$ are any symbol of at least $C^r_*$ regularity in $x$, define
$$
(a\#_{r;q} b)(x,\xi):=\sum_{\alpha:|\alpha|\leq r}\Df_{q,\xi}^\alpha a(x,\xi)\cdot X_{q,x}^{(\alpha)} b(x,\xi).
$$
\end{definition}

We now state the first main theorem of this section.

\begin{theorem}[Composition of Para-differential Operators, I]\label{Compo1}
Suppose $r>0$, $m,m'$ are real numbers. Let $q$ be a RT-admissible tuple, whose components are linear combinations of the fundamental tuple $Q$ of $\Gp$, and let $X_q^{(\alpha)}$ be the left-invariant differential operators as in Proposition \ref{TaylorGp}. Given $a\in \mathscr{S}^m_{1,1}$, $b\in\mathcal{A}_r^{m'}$, for an admissible cut-off function $\chi$ with parameter $\delta$ sufficiently small (depending on the magnitude of $r$ and the algebraic structure of $\Gp$ only),
$$
\Op(a)\circ\Op(b^\chi)
-\Op(a\#_{r;q} b^\chi)\in \Op \mathscr{S}^{m+m'-r}_{1,1}(\Gp).
$$
More precisely, the operator norm of $\Op(a)\circ\Op(b^\chi)
-\Op(a\#_{r;q} b^\chi)$ for $H^{s+m+m'-r}\to H^s$ is bounded by
$$
C_s\mathbf{M}^{m}_{l,n+2;q}(a)\mathbf{W}^{m';r}_{l;q}(b),
$$
where the integer $l$ does not depend on $a,b$.
\end{theorem}
\begin{proof}
We observe a simple fact: for symbols $\sigma_1,\sigma_2$ of type (1,1), the symbol of $\Op(\sigma_1)\circ\Op(\sigma_2)$ is
\begin{equation}\label{OpaOpb}
\int_{\Gp}\mathcal{K}_{\sigma_1}(x,y)\xi^*(y)\sigma_2(xy^{-1},\xi)dy,
\end{equation}
where $\mathcal{K}_{\sigma_1}$ is the convolution kernel of $\Op(\sigma_1)$, defined by (\ref{Symbol-Ker})-(\ref{Ker-Conv}).

Let us now examine the symbol $\sigma$ of $\Op(a)\circ\Op(b^\chi)$. We start from $b^\chi(x,\xi)$: since $b^\chi(x,\xi)$ satisfies the spectral condition with parameter $\delta$, we can choose another admissible cut-off function $\chi_1(\mu,\lambda)$ with parameter $2\delta$, and find 
$$
\Ft{b^\chi}(\eta,\xi)=\chi_1(|\eta|,|\xi|)\Ft{b^\chi}(\eta,\xi.
$$
Converting back to the physical space, this implies
$$
b^\chi(x,\xi)=\int_\Gp \check{\chi}_1(z^{-1}x,|\xi|)b^\chi(z)dz.
$$
Here with a little abuse of notation, we write $\check{\chi}_1(y,|\xi|)$ for the Fourier inversion of $\chi_1(|\eta|,|\xi|)\I[\eta]$ with respect to $\eta$. Note that $\chi_1(|\eta|,|\xi|)\I[\eta]$ is a scaling on each representation space $\Hh[\eta]$, so the order of convolution is irrelevant. By formula (\ref{OpaOpb}), 
$$
\begin{aligned}
\sigma(x,\xi)
&=\int_{\Gp}\mathcal{K}_{a}(x,y)\xi^*(y)\left(\int_\Gp \check{\chi}_1(z^{-1}xy^{-1},|\xi|)b^\chi(z,\xi)dz\right)dy\\
&=\int_{\Gp}\mathcal{K}_{a}(x,y)\xi^*(y)\left(\int_\Gp \check{\chi}_1(zy^{-1},|\xi|)b^\chi(xz^{-1},\xi)dz\right)dy.
\end{aligned}
$$
Using Proposition \ref{TaylorGp} with $N$ being the smallest integer $>r$, we compute the inner integral as 
$$
\begin{aligned}
\int_\Gp \check{\chi}_1(zy^{-1},|\xi|)b^\chi(xz^{-1},\xi)dz
&=\sum_{|\alpha|\leq r}\left(\int_\Gp \check{\chi}_1(zy^{-1},|\xi|)q^\alpha(z)dz\right)X^{(\alpha)}_{q,x}b^\chi(x,\xi)\\
&\quad+\int_\Gp \check{\chi}_1(zy^{-1},|\xi|)R_N\big(b^\chi(\cdot,\xi);x,z^{-1}\big)dz
\end{aligned}
$$
The integrals in the sum $\sum_{|\alpha|\leq r}$ are of convolution type. By Proposition \ref{SpecPrd0}, the Fourier support of $q^\alpha$ for $|\alpha|\leq r$ is within $\mathfrak{S}[0,C]$, where $C$ depends on $r$ and the algebraic structure of $\Gp$ only. Thus, for $|\xi|\geq 2C/\delta$, we have
$$
\int_\Gp \check{\chi}_1(zy^{-1},|\xi|)q^\alpha(z)dz=q^\alpha(y),
$$
and consequently,
\begin{equation}\label{sigmaTemp1}
\begin{aligned}
\sigma(x,\xi)
&=\sum_{|\alpha|\leq r}\int_{\Gp}\mathcal{K}_{a}(x,y)\xi^*(y)q^\alpha(y)X^{(\alpha)}_{q,x}b^\chi(x,\xi)dy\\
&\quad+\int_{\Gp}\mathcal{K}_{a}(x,y)\xi^*(y)
\left(\int_\Gp \check{\chi}_1(zy^{-1},|\xi|)R_N\big(b^\chi(\cdot,\xi);x,z^{-1}\big)dz\right)dy\\
&=:\sum_{|\alpha|\leq r}\Df_{q,\xi}^\alpha a(x,\xi)\cdot X_{q,x}^{(\alpha)}b^\chi(x,\xi)+\varsigma_N(x,\xi).
\end{aligned}
\end{equation}

The sum in the right-hand-side of (\ref{sigmaTemp1}) is just $a\#_{r;q}b^\chi$, which is in $\mathscr{S}^{m+m'}_{1,1}$. There is then only one assertion to be verified: the term $\varsigma_N(x,\xi)$ is of class $\mathscr{S}^{m+m'-r}_{1,1}$, from which it follows that the symbol $\sigma$ really is of class $\mathscr{S}^{m+m'}_{1,1}$. Note that the lower frequency terms give smoothing symbols in $\mathscr{S}^{-\infty}$, so we just have to verify the assertion for $|\xi|\geq 2C/\delta$. 

Let us first show that $\varsigma_N(x,\xi)\lesssim\size[\xi]^{m+m'-r}$. We can change the order of integration as follows:
\begin{equation}\label{sigmaTemp2}
\begin{aligned}
\varsigma_N(x,\xi)
=\int_{\Gp}
\left(\int_\Gp \mathcal{K}_a(x,y)\xi^*(y)\check{\chi}_1(zy^{-1},|\xi|)dy\right)R_N\big(b^\chi(\cdot,\xi);x,z^{-1}\big)dz.
\end{aligned}
\end{equation}
Just as in the proof of Stein's theorem, we decompose $\mathcal{K}_a=\sum_{j=0}^\infty \mathcal{K}_{a_j}$, where $\mathcal{K}_{a_j}$ is the convolution kernel of $a_j(x,\xi):=a(x,\xi)\vartheta_j\big(|\xi|\big)$, with $\vartheta$ as in (\ref{LPDisc}). For the inner integral with respect to $y$, each summand
$$
P_j(x,z):=\int_\Gp \mathcal{K}_{a_j}(x,y)\xi^*(y)\check{\chi}_1(zy^{-1},|\xi|)dy
$$
is of convolution type. By spectral property i.e. Corollary \ref{SpecPrd}, the Fourier support of $\mathcal{K}_{a_j}(x,y)\xi^*(y)$ with respect to $y$ is in the set 
$$
\mathfrak{S}\Big[c\big|2^j-|\xi|\big|,2^j+|\xi|\Big].
$$
Hence $\Ft{P_j}(x,\eta)$ (the Fourier transform is taken with respect to $z$) does not vanish only for 
$$
c\big|2^j-|\xi|\big|\leq 2\delta\size[\xi].
$$
If we choose $\delta\ll c$, this implies that $2^j$ must be comparable with $\size[\xi]$. As a result,
$$
\varsigma_N(x,\xi)
=\sum_{j:2^j\simeq\size[\xi]}\int_{\Gp}\mathcal{K}_{a_j}(x,y)\xi^*(y)
\left(\int_\Gp \check{\chi}_1(zy^{-1},|\xi|)R_N\big(b^\chi(\cdot,\xi);x,z^{-1}\big)dz\right)dy.
$$
Just as inequalitites (\ref{SteinIneq1})-(\ref{SteinIneq2}) in the proof of Stein's theorem, the estimate
\begin{equation}\label{sigmaTemp3}
\big(1+|2^j\dist(y,e)|^2\big)^{L/2}\cdot|\mathcal{K}_{a_j}(x,y)|
\lesssim\mathbf{M}^m_{0,L;Q}(a)\cdot2^{jm+jn}
\end{equation}
is valid for $L\in\mathbb{N}$. The remainder in Taylor's formula can be estimated using the remark following Proposition \ref{ParaDiff1}: since $N>r$, we have, by the improved growth estimate (Proposition \ref{ParaDiff1}),
$$
\begin{aligned}
\big\|R_N\big(b^\chi(\cdot,\xi);x,z^{-1}\big)\big\|
&\lesssim \|b^\chi(x,\xi)\|_{C^N_x;\mathrm{End}(\Hh[\xi])}\cdot\dist(z,e)^N \\
&\lesssim \mathbf{W}^{m;r}_{0;Q}(b)\cdot\size[\xi]^{m'+N-r}\dist(z,e)^N.
\end{aligned}
$$
Here the norm is the operator norm on $\mathrm{End}(\Hh[\xi])$, and the implicit constants do not depend on $\xi$. By Corollary \ref{ConvVanish}, noting the integrand is smooth and equals $O(\dist(z,e)^N)$, we thus estimate
\begin{equation}\label{sigmaTemp4}
\begin{aligned}
\int_\Gp \check{\chi}_1(zy^{-1},|\xi|)R_N\big(b^\chi(\cdot,\xi);x,z^{-1}\big)dz
=\sum_{k=0}^N\upsilon_k(x,y,\xi),
\end{aligned}
\end{equation}
where each $\upsilon_k(x,y,\xi)$ is smooth in $x$, and satisfies
$$
\|\upsilon_k(x,y,\xi)\|
\lesssim\mathbf{W}^{m;r}_{0;Q}(b)\cdot\size[\xi]^{m'+k-r}\dist(y,e)^k.
$$
Combining every inequality (\ref{sigmaTemp3})-(\ref{sigmaTemp4}), using Lemma \ref{Fisqq'Lem}, we estimate the operator norm of $\varsigma_N(x,\xi)$ as (noting that $\|\xi^*(y)\|\equiv1$)
$$
\begin{aligned}
\|\varsigma_N(x,\xi)\|
&\lesssim \sum_{j:2^j\simeq\size[\xi]}\sum_{k=0}^N
\left\|\int_{\Gp}\mathcal{K}_{a_j}(x,y)\xi^*(y)\upsilon_k(x,y,\xi)dy\right\|\\
&\lesssim \sum_{k=0}^N\mathbf{M}^m_{0,n+1;Q}(a)\cdot\mathbf{W}^{m;r}_{0;Q}(b)\cdot\size[\xi]^{m'+k-r}
\sum_{j:2^j\simeq\size[\xi]}2^{jm+jn}
\int_{\Gp}\frac{\min\left(\dist(y,e)^k,1\right)dy}{\big(1+|2^j\dist(y,e)|^2\big)^{(n+1)/2}}\\
&\lesssim \mathbf{M}^m_{0,n+1;Q}(a)\cdot\mathbf{W}^{m;r}_{0;Q}(b)\cdot\size[\xi]^{m+m'-r}.
\end{aligned}
$$
This shows that the second sum in (\ref{sigmaTemp1}) is controlled by $\size[\xi]^{m+m'-r}$.

Finally, we use induction on the order of differentiation in (\ref{OpaOpb}). Noticing that $\Df_{Q^\beta,\xi}\big(\xi^*(y)\big)=Q^\beta(y)\xi^*(y)$, we have, for example, given any $X\in\mathfrak{g}$,
\begin{equation}\label{sigmaTemp5}
X_x\sigma(x,\xi)
=\int_{\Gp}X_x\mathcal{K}_a(x,y)\xi^*(y)b^\chi(xy^{-1},\xi)dy
+\int_{\Gp}\mathcal{K}_a(x,y)\xi^*(y)X_x\big(b^\chi(xy^{-1},\xi)\big)dy
\end{equation}
and by Leibniz type property of the fundamental tuple $Q$,
\begin{equation}\label{sigmaTemp6}
\begin{aligned}
\Df_{Q_i,\xi}\sigma(x,\xi)
&=\int_{\Gp}\mathcal{K}_a(x,y)Q_i(y)\xi^*(y)b^\chi(xy^{-1},\xi)dy+\int_{\Gp}\mathcal{K}_a(x,y)\xi^*(y)(\Df_{Q_i,\xi}b^\chi)(xy^{-1},\xi)dy\\
&\quad+\sum_{j,k}c_{j,k}^i\int_{\Gp}\mathcal{K}_a(x,y)Q_j(y)\xi^*(y)(\Df_{Q_k,\xi}b^\chi)(xy^{-1},\xi)dy.
\end{aligned}
\end{equation}
Each of the above integrals can be estimated in exactly the same method, with $m$ replaced by $m+1$ in (\ref{sigmaTemp5}) and by $m-1$ in (\ref{sigmaTemp6}); we just have to observe that $X_x\mathcal{K}_a(x,y)$ is the convolution kernel of $X_xa(x,\xi)$, and $\mathcal{K}_a(x,y)Q_i(y)$ is the convolution kernel of $\Df_{Q_i,\xi}a(x,\xi)$.
\end{proof}

Now if $\Op(a)$ is in fact a para-differential operator corresponding to some rough symbol of class $\mathcal{A}^m_r$, we can actually strengthen Theorem \ref{Compo1}, and deduce the composition formula of para-differential operators: 
\begin{theorem}[Composition of Para-differential Operators, II]\label{Compo2}
Suppose $r>0$, $m,m'$ are real numbers. Let $q$ be a RT-admissible tuple, whose components are are linear combinations of the fundamental tuple $Q$ of $\Gp$, and let $X_q^{(\alpha)}$ be the left-invariant differential operators as in Proposition \ref{TaylorGp}. Given $a\in \mathcal{A}_r^{m}$, $b\in\mathcal{A}_r^{m'}$, it follows that
$$
T_a\circ T_b-T_{a\#_{r,q}b}
\in\Op\Sigma_{<1/2}^{m+m'-r}.
$$
More precisely, the operator norm of $T_a\circ T_b-T_{a\#_{r,q}b}$ for $H^{s+m+m'-r}\to H^s$ is bounded by
$$
C_s\mathbf{W}^{m;r}_{l;q}(a)\mathbf{W}^{m';r}_{l;q}(b),
$$
where the integer $l$ does not depend on $a,b$.
\end{theorem}

Before we prove Theorem \ref{Compo2}, we state two auxiliary results.
\begin{lemma}\label{CompoAux1}
Suppose $r>0$, $m,m'$ are real numbers, and $a\in \mathcal{A}_r^{m}$, $b\in\mathcal{A}_r^{m'}$. There are constants $\delta_0\in(0,1/2)$ and $C>0$ depending only on the algebraic structure of $\Gp$ with the following property: if $\chi$ is an admissible cut-off function with parameter $\delta<\delta_0$, the symbol of $T_a^\chi\circ T_b^\chi$ is of class $\Sigma^{m+m'-r}_{C\delta}$.
\end{lemma}
\begin{lemma}\label{CompoAux2}
There is a constant $C>1$, depending on the algebraic structure of $\Gp$ only, with the following property. Let $A$ be a finite dimensional normed algebra. Suppose $0\leq r_1\leq r_2$, with $r_2>0$. Suppose $u\in C^{r_1}_*$, $v\in C^{r_2}_*$ are $A$-valued functions. Suppose $\vartheta\in C_0^\infty(\mathbb{R})$ is such that $\vartheta(\lambda)=1$ for $|\lambda|\leq1$ and vanishes for $|\lambda|\geq2$. Then for $t\geq1$,
$$
\left\|\vartheta\left(\frac{|\nabla|}{Ct}\right)(uv)-\vartheta\left(\frac{|\nabla|}{t}\right)u\cdot \vartheta\left(\frac{|\nabla|}{t}\right)v\right\|_{L^\infty;A}
\lesssim_{r_1,r_2}t^{-r_1}\|u\|_{C^{r_1}_*;A}\|v\|_{C^{r_2}_*;A}.
$$
The implicit constants do not depend on $A$.
\end{lemma}
\begin{proof}[Proof of Lemma \ref{CompoAux1}]
By formula (\ref{OpaOpb}), the symbol $\sigma$ of $T_a^\chi\circ T_b^\chi$ is 
$$
\begin{aligned}
\sigma(x,\xi)
&=\int_{\Gp}\mathcal{K}_{a^\chi}(x,y)\xi^*(y)b^\chi(xy^{-1},\xi)dy\\
&=\sum_{\zeta\in\DuGp}d_\zeta\int_\Gp\Tr\Big[a^\chi(x,\zeta)\zeta(y)\Big]\xi^*(y)b^\chi(xy^{-1},\xi)dy.
\end{aligned}
$$
Let us fix $\xi\in\DuGp$. By the spectral localization property of products (Corollary \ref{SpecPrd}), the Fourier support of $\zeta\otimes\xi^*$ is in 
$$
\mathfrak{S}\Big[c\big||\zeta|-|\xi|\big|,|\zeta|+|\xi|\Big],
$$
where $0<c<1$ depends on the algebraic structure of $\Gp$ only. Since the integral is of convolution type, we find that the Fourier support of $\int_\Gp (\zeta\otimes\xi)^*(y)b^\chi(xy^{-1},\xi)dy$ with respect to $x$ is contained in the set
$$
\mathfrak{S}\Big[c\big||\zeta|-|\xi|\big|,|\zeta|+|\xi|\Big]\cap \mathfrak{S}[0,\delta\size[\xi]],
$$
so the integral does not vanish only for $|\zeta|\leq C\size[\xi]$, where $C$ depends on the algebraic structure of $\Gp$ only. Now applying the spectral condition satisfied by $a^\chi(x,\zeta)$, using the spectral localization property of products again, we obtain that the Fourier support of $\sigma(x,\xi)$ is contained in the set $\mathfrak{S}[0,C\delta\size[\xi]]$. That $\sigma\in\mathscr{S}^{m+m'}_{1,1}$ for sufficiently small $\delta$ is a direct consequence of Theorem \ref{Compo1}.
\end{proof}
The proof of Lemma \ref{CompoAux2} is identical to that of Proposition 8.6.9. of \cite{Hormander1997}, with the application of Corollary \ref{SpecPrd}, Corollary \ref{LPZygCor} and its vector-valued version as described in Remark \ref{VectZyg}.

\begin{proof}[Proof of Theorem \ref{Compo2}]
We notice a simple fact: if $q$ is an RT-admissible tuple and $X_q^{(\alpha)}$ are the corresponding left-invariant differential operators as in Proposition \ref{TaylorGp}, then for $a\in\mathcal{A}^m_r$, $b\in\mathcal{A}^{m'}_r$, it follows that $a\#_{r;q}b$ is an element in $\bigcup_{j\leq r}\mathcal{A}_{r-j}^{m+m'-j}$ since by the product rule in Zygmund spaces,
\begin{equation}\label{Prd_r}
\sum_{|\alpha|=j}\Df_{q,\xi}^\alpha a(x,\xi)\cdot X_{q,x}^{(\alpha)} b(x,\xi)\in\mathcal{A}_{r-j}^{m+m'-j}.
\end{equation}

First of all, we need to justify the notation that disregards the admissible cut-off function. In fact, Lemma \ref{Sigmadelta}, Proposition \ref{FreedomForCut-off} and formula (\ref{Prd_r}) ensures that $T_{a\#_{r;q}b}$ is defined modulo $\Op\Sigma_{<1/2}^{m+m'-r}$, provided that the parameter of the admissible cut-off function is sufficiently small (depending on the algebraic structure of $\Gp$ only). We also compute, for admissible cut-off functions $\chi_1,\chi_2$ with parameter $\delta_1<\delta_2$ respectively,
$$
\Op(a^{\chi_1})\circ\Op(b^{\chi_1})-\Op(a^{\chi_2})\circ\Op(b^{\chi_2})
=\Op(a^{\chi_1}-a^{\chi_2})\circ\Op(b^{\chi_1})+\Op(a^{\chi_2})\circ\Op(b^{\chi_1}-b^{\chi_2}).
$$
By Proposition \ref{FreedomForCut-off}, $a^{\chi_1}-a^{\chi_2}\in\Sigma_{\delta_2}^{m-r}$, $b^{\chi_1}-b^{\chi_2}\in\Sigma_{\delta_2}^{m'-r}$. Thus if $\delta_1,\delta_2$ are sufficiently small (depending on the algebraic structure of $\Gp$ only), Lemma \ref{CompoAux1} then implies that the right-hand-side is in $\Op\Sigma_{C\delta_2}^{m+m'-r}\subset\Op\Sigma_{<1/2}^{m+m'-r}$. Thus $T_a\circ T_b$ is defined modulo $\Op\Sigma_{<1/2}^{m+m'-r}$. Applying Lemma \ref{Sigmadelta} and Theorem \ref{Compo1} once again, we find that if $\chi$ is an admissible cut-off function with sufficiently small parameter $\delta$ (depending on the algebraic structure of $\Gp$ only), then
$$
T_a^\chi\circ T_b^\chi-\Op\big(a^\chi\#_{r;q}b^\chi\big)
\in\Op\Sigma_{<1/2}^{m+m'-r}.
$$
From now on, we fix $\chi$ to be the function in (\ref{AdmCutoffdelta}).

It remains to show that $a^\chi\#_{r;q}b^\chi$ equals $(a\#_{r;q}b)^{\chi_1}$ modulo $\Op\Sigma_{<1/2}^{m+m'-r}$ for some suitable admissible cut-off function $\chi_1$. For simplicity we write $a^{(\alpha)}=\Df_{q,\xi}^\alpha a$, $b_{(\alpha)}=X_{q,x}^{(\alpha)} b$. We just need to show that for $|\alpha|\leq r$, each summand of $(a^\chi \#_r b^\chi )-(a\#_{r;q}b)^{\chi_1}$, being
$$
(a^\chi)^{(\alpha)}\cdot(b^\chi)_{(\alpha)}
-\left(a^{(\alpha)}\cdot b_{(\alpha)}\right)^{\chi_1},
$$
is in $\Sigma^{m+m'-r}_{<1/2}$. The only technical complexity comes from $(a^\chi)^{(\alpha)}$. In fact, from the proof of Proposition \ref{a-a^chi}, we know that
$$
\Df_{Q_i,\xi}a^\chi(x,\xi)
=(\Df_{Q_i,\xi}a)^\chi(x,\xi)
+\sum_{j,k=1}^{M_i}c^i_{j,k}\Df_{Q_j,\xi}(\chi\cdot\I[\xi])\big(|\nabla_x|,|\xi|\big)\Df_{Q_k}a(x,\xi),
$$
and belongs to $\Sigma_{C\delta}^{m-1}$ for some $C$ depending on the algebraic structure of $\Gp$ only. Thus, if $\delta$ is sufficiently small, we obtain by induction that
$$
(a^\chi)^{(\alpha)}\cdot(b^\chi)_{(\alpha)}=(a^{(\alpha)})^\chi\cdot(b_{(\alpha)})^\chi\,\mod\,\Sigma_{<1/2}^{m+m'-r},
$$
since $|\alpha|\leq r$. 

The claim is then reduced to checking
\begin{equation}\label{2summand}
(a^{(\alpha)})^\chi\cdot(b_{(\alpha)})^\chi
-\left(a^{(\alpha)}\cdot b_{(\alpha)}\right)^{\chi_1}
\in \Sigma^{m+m'-r}_{<1/2}.
\end{equation}
To verify this, we just need the technical Lemma \ref{CompoAux2}. Recall that we chose $\chi$ as in (\ref{LPCont}). We regard $\xi$ as fixed, then set $\chi_1(|\eta|,|\xi|)=\chi(|\eta|/C,|\xi|)$ where $C$ is as in Lemma \ref{CompoAux2}, and $\vartheta(|\eta|)=\chi_1(|\eta|,|\xi|)$. Since $a^{(\alpha)}\in \mathcal{A}_{r}^{m-|\alpha|}$, $b_{(\alpha)}\in \mathcal{A}_{r-|\alpha|}^{m}$, Lemma \ref{CompoAux2} gives
$$
\begin{aligned}
\left\|(a^{(\alpha)})^\chi\cdot(b_{(\alpha)})^\chi
-\left(a^{(\alpha)}\cdot b_{(\alpha)}\right)^{\chi_1}\right\|
&\lesssim 
\langle\xi\rangle^{-(r-|\alpha|)}\cdot\mathbf{W}^{m-|\alpha|;r}_{0;Q}(a)\langle\xi\rangle^{m-|\alpha|}
\cdot\mathbf{W}^{m';r-|\alpha|}_{0;Q}(b)\langle\xi\rangle^{m'}\\
&\lesssim \mathbf{W}^{m-|\alpha|;r}_{0;Q}(a)\cdot\mathbf{W}^{m';r-|\alpha|}_{0;Q}(b)\cdot\langle\xi\rangle^{m+m'-r}.
\end{aligned}
$$
Here the norm to be estimated is the operator norm on $\mathrm{End}(\Hh[\xi])$, and the implicit constants do not depend on $\xi$. Derivatives of (\ref{2summand}) can be estimated similarly. By Lemma \ref{Sigmadelta}, (\ref{2summand}) still satisfies a spectral condition. This completes the proof. 
\end{proof}

As a corollary, we obtain the commutator property of para-differential operators: the commutator of two para-differential operators of order $m$ and $m'$ gives rise to a para-differential operator of order $m+m'-1$. 

\begin{corollary}\label{ParaComm}
Suppose $r>0$, $m,m'$ are real numbers.  Given $a\in \mathcal{A}_r^{m}$, $b\in\mathcal{A}_r^{m'}$, the commutator $[T_a,T_b]$ is in the class $\Op\Sigma_{<1/2}^{m+m'-1}$. More precisely, the operator norm of $[T_a,T_b]$ for $H^{s+m+m'-1}\to H^s$ is bounded by
$$
C_s\mathbf{W}^{m;r}_{l;q}(a)\mathbf{W}^{m';r}_{l;q}(b),
$$
where the integer $l$ does not depend on $a,b$.
\end{corollary}
\begin{proof}
Imitating the proof of Proposition \ref{2OrderComm}, we find that the commutator symbol $[a,b](x,\xi)$ is of class $\mathcal{A}^{m+m'-1}_r$; we just need to repeat the proof of Proposition \ref{2OrderComm} and, in addition, estimate the $C^r_*$ norms with respect to $x$.

On the other hand, 
$$
\begin{aligned}
(a\#_{r;q} b)(x,\xi)-(b\#_{r;q} a)(x,\xi)
&=[a,b](x,\xi)\\
&\quad+\sum_{\alpha:1\leq|\alpha|\leq r}\left(\Df_{q,\xi}^\alpha a(x,\xi)\cdot X_{q,x}^{(\alpha)} b(x,\xi)
-\Df_{q,\xi}^\alpha b(x,\xi)\cdot X_{q,x}^{(\alpha)} a(x,\xi)\right).
\end{aligned}
$$
The sum is easily seen to give rise to symbols of order $m+m'-1$. Theorem \ref{Compo2} then shows that
$$
[T_a,T_b]
\in\Op\Sigma^{m+m'-1}_{<1/2}.
$$
\end{proof}

We can compare this with \cite{Delort2015}: in order to cast the normal form reduction for quasi-linear Hamiltonian Klein-Gordon equation, such commutator property is necessary. The definition of para-differential operator within out formalism is of course different from that in \cite{Delort2015}. While it is not known whether these two definitions coincide, this does not affect our goal, as a calculus is already available.

We also state the adjoint formula for para-differential operators and omit the proof.
\begin{theorem}[Adjoint of Para-differential Operator]\label{ParaAdj}
Suppose $r>0$, $m$ is a real number. Let $q$ be a RT-admissible tuple, whose components are from the fundamental tuple $Q$ of $\Gp$, and let $X_q^{(\alpha)}$ be the left-invariant differential operators as in Proposition \ref{TaylorGp}. Given $a\in \mathcal{A}_r^{m}$, the adjoint operator $T_a^*$ satisfies
$$
T_a^*-T_{a^{\bullet;r,q}}\in\Op\Sigma^{m-r}_{<1/2}
$$
where
$$
a^{\bullet;r,q}(x,\xi)=\sum_{|\alpha|\leq r}\left(\Df_{q,\xi}^\alpha X_{q,x}^{(\alpha)}a^*\right)(x,\xi).
$$
More precisely, the operator norm of $[T_a,T_b]$ for $H^{s+m+m'-1}\to H^s$ is bounded by
$$
C_s\mathbf{W}^{m;r}_{l;q}(a),
$$
where the integer $l$ does not depend on $a$.
\end{theorem}

\subsection{Symbols with Very Rough Regularity}
Sometimes it is also necessary to discuss symbols $a(x,\xi)$ which are merely $C^{-r}_*$ in $x$ with $r>0$. We have the following proposition:
\begin{proposition}\label{T_aNegIndex}
Suppose $r>0$. Fix a basis $X_1,\cdots,X_n$ of $\mathfrak{g}$, and define $X^\alpha$ as in Proposition \ref{NormalOrder}. Suppose $a\in\mathcal{A}_{-r}^m(\Gp)$\footnote{The definition of symbol class $\mathcal{A}_{r}^m(\Gp)$ of course can be directly extended beyond $r>0$.} and $\chi$ is an admissible cut-off function with parameter $\delta$. Then $a^\chi(x,\xi)\in\Sigma^{m+r}_{<1/2}$. In fact, for any strongly RT-admissible tuple $q$, we have
$$
\big\|X_x^\alpha\Df_{q,\xi}^\beta a^\chi (x,\xi)\big\|
\lesssim_{\alpha,\beta;q}
\mathbf{W}^{m;r}_{|\beta|;q}(a)\cdot\size[\xi]^{m+r+|\alpha|-|\beta|}.
$$
\end{proposition}
The proof is a mere application of the growth estimate indicated in Remark \ref{Zygmund<0}. Since such symbols are quite irrregular, we cannot expect that choosing a different admissible cut-off function would produce a negeligible error. But we may still refer $\Op(a^\chi)$ as a \emph{para-differential operator} $T_a$, which maps $H^{s}$ to $H^{s+r}$ continuously for any $s\in\mathbb{R}$. Fortunately, there is a simple trick to bypass its low regularity and incorporate it into the symbolic calculus that we just constructed. For example, if $0<r<1$ and $a\in\mathcal{A}^m_{-r}$, then we consider $\Delta^{-1}a$ (subtracting the mean value to make this well-defined), and simply notice that
$$
T_af
=\big[\Delta,T_{\Delta^{-1}a}\big] f-2T_{\nabla \Delta^{-1}a\cdot\nabla}f.
$$
The right-hand-side involves only regular symbols. By the formula of compositions we just proved, noting that $\Delta^{-1}a\in\mathcal{A}^{m+2}_{2-r}$, we find
$$
\big[\Delta,T_{\Delta^{-1}a}\big]
-2T_{\nabla \Delta^{-1}a\cdot\nabla}
$$
is in fact a para-differential operator of order $m+2-(2-r)=m+r$. 

\subsection{Analytic Function of Symbols and Para-Differential Operators}
In this subsection, we claim several propositions concerning analytic function of para-differential operators, to whose proof we apply the results obtained so far. They will be constantly needed for symbolic calculus in future applications.

Although a calculus is already available for our para-differential operators, uncertainty still appears when dealing with commutators. More precisely, if $a,b$ are only general rough symbols on $\Gp$ with order $m$ and $m'$ (in the sense of definition \ref{2Order}), then nothing is known for $[a,b]$ besides that it is a symbol of order $m+m'-1$. Such issue is of course trivially resolved in the commutative group case, but for a general compact Lie group, it would be feasible to look for a smaller subclass of symbols which possesses better commutator properties.

We start with function of symbols. The difficulty of non-trivial commutator can be partially resolved for a special class of symbols large enough for our application. Roughly speaking, such class consists of symbols of ``homogenized" classical differential operators. The prototype of such symbols is
$$
\sqrt{|\xi|^2+b(x,\xi)},
\quad 
b \text{ is the symbol of a vector field},
$$
i.e. the square-root of perturbed Laplacian. In the commutative group case, symbols like this may be manipulated as usual scalar-valued functions. We will show that such manipulation remains partially valid on our group $\Gp$.

\begin{definition}\label{QuasiHomoSym}
A quasi-homogeneous symbol of order $m$ on $\Gp$ takes the form
$$
\kappa(\xi)^mf\left(\frac{b(x,\xi)}{\kappa(\xi)}\right).
$$
Here the scalar Fourier multiplier $\kappa:\DuGp\to(0,+\infty)$ is the symbol of $|\nabla|=\sqrt{-\Delta}$, and $b(x,\xi)$ is the symbol of a vector field on $\Gp$, and the Borel function $f:\mathbb{C}\to\mathbb{C}$ is bounded. The action $f$ on each $\mathrm{End}(\Hh[\xi])$ is defined by spectral calculus of matrices.
\end{definition}

We first need some properties concerning the difference operator acting on the scalar symbol $\kappa(\xi)$. For simplicity, we fix a basis $\{X_i\}_{i=1}^n$ of $\mathfrak{g}$, so that $b(x,\xi)=\sum_{i=1}^nb_i(x)\sigma[X_i](\xi)$; and write $\{\Df_\mu\}$ for the RT-difference operators corresponding to the fundamental representations of $\Gp$. Throughout the rest of this subsection, summation with respect to Greek indices is interpreted as summation over these difference operators.
\begin{proposition}\label{Dkappa}
Let $\kappa$ be the symbol of $|\nabla|$ on $\Gp$.

(1) There necessarily holds $\kappa\in\mathscr{S}^1_{1,0}$, and $\Df_\mu \kappa$ commutes with every symbol of order $m$ up to an error of order $m-1$. Here the notion of order is as in Definition \ref{2Order}.

(2) For every real number $m\in\mathbb{R}$, the symbol $\Df_\mu\kappa^m-m\kappa^{m-1}\Df_\mu\kappa\in\mathscr{S}^{m-2}_{1,0}$.

(2) Furthermore, for any natural number $n$, there holds
$$
\left\|\Df_\mu\big(\kappa(\xi)^{-n}\big)-\frac{n\Df_\mu\kappa(\xi)}{\kappa(\xi)^{n+1}}\right\|
\lesssim \frac{n^2\big(1+C|\xi|^{-1}\big)^{n+1}}{|\xi|^{n+2}}.
$$
with the constants independent from $n,\xi$. Here as usual the norms are operator norms on $\Hh[\xi]$.
\end{proposition}
\begin{proof}
(1) The claim $\kappa\in\mathscr{S}^1_{1,0}$ follows from two premises: $|\nabla|$ is of class $\Psi^1_{1,0}(\Gp)$, the H\"{o}rmander class; the operator class $\Psi^m_{1,0}(\Gp)$ coincides with $\Op\mathscr{S}^m_{1,0}(\Gp)$ by Theorem \ref{Fischer}. We next notice that $\kappa$ commutes with any symbol $a$; thus by the Leibniz property,
$$
0=a\Df_\mu\kappa-\Df_\mu\kappa a
+\sum_{\nu,\nu'} c_{\mu\nu\nu'}\left(\Df_\nu a\Df_{\nu'}\kappa-\Df_\nu \kappa\Df_{\nu'}a\right).
$$
If $a$ is of order $m$, this obviously implies that $a\Df_\mu\kappa-\Df_\mu\kappa a$ is of order $m-1$.

(2) We shall extensively utilize Theorem \ref{Fischer}, namely $\Op\mathscr{S}^m_{1,0}=\Psi^m_{1,0}(\Gp)$. Since the Greek letter $\xi$ has already occupied the notation for representations, we use $(x,\omega)$ to denote elements in the cotangent bundle $\mathrm{T}^*\Gp$, instead of the commonly used $(x,\xi)$ in most literature. 

Consider any smooth function $b$ on $\Gp$, and then the commutator $\big[|\nabla|^m,b\big]\in\Psi^{m-1}_{1,0}$. By the global symbolic calculus formula, i.e. Theorem \ref{RegCompo}, we find that the invariant symbol of $\big[|\nabla|^m,b\big]$ is
$$
\sum_\mu \Df_\mu\kappa(\xi)^mX_\mu b(x)
\quad \mod\mathscr{S}^{m-2}_{1,0}.
$$
Here $\{X_\mu\}$ denotes the set of left-invariant vector fields adapting to the RT-difference operators $\Df_\mu$. In the special case $m=2$, this is just
$$
2\sum_\mu \kappa(\xi)\Df_\mu\kappa(\xi)X_\mu b(x)
\quad \mod\mathscr{S}^{0}_{1,0}.
$$
Since we have the precise equality $[\Delta,b]=2\nabla b\cdot\nabla+\Delta b$, we obtain
$$
\Op\left(\sum_\mu \kappa(\xi)\Df_\mu\kappa(\xi)X_\mu b(x)\right)
=-\nabla b\cdot\nabla
\quad \mod\Psi^{0}_{1,0}.
$$
By applying $m|\nabla|^{m-2}$ on both sides, we find, by the formula of global symbolic calculus once again,
$$
\begin{aligned}
\Op\left(\sum_\mu m\kappa(\xi)^{m-1}\Df_\mu\kappa(\xi)X_\mu b(x)\right)
&=m|\nabla|^{m-2}\big(\nabla b\cdot\nabla\big)
&\mod\Psi^{m-2}_{1,0}\\
&=m\nabla b\cdot\nabla|\nabla|^{m-2} 
&\mod\Psi^{m-2}_{1,0}.
\end{aligned}
$$

But in the formalism of H\"{o}rmander calculus, the commutator $\big[|\nabla|^m,b\big]$ has \emph{principal symbol} just being the Poisson bracket between $|\omega|_g^m$ and $b(x)$, namely
$$
\sum_{j=1}^n\partial_{\omega_j}|\omega|_{g}^m\partial_{x^j}b
=\sum_{j=1}^nm|\omega|_{g}^{m-2}{\omega_j}\partial_{x^j}b,
$$
where the local trivialization of the cotangent bundle is arbitrary, and $g$ is the bi-invariant metric on $\Gp$. Noting that the principal symbol of $\nabla b\cdot\nabla$ is just $\sum_{j=1}^n\omega_j\partial_{x^j}b$, this implies
$$
\big[|\nabla|^m,b\big]
=m\nabla b\cdot\nabla|\nabla|^{m-2}
\quad\mod\Psi^{m-2}_{1,0}.
$$
In conclusion, we find
$$
\Op\left(\sum_\mu m\kappa(\xi)^{m-1}\Df_\mu\kappa(\xi)X_\mu b(x)\right)
=\big[|\nabla|^m,b\big]
\quad\mod\Psi^{m-2}_{1,0},
$$
or in other words,
$$
\sum_\mu m\kappa(\xi)^{m-1}\Df_\mu\kappa(\xi)X_\mu b(x)
=\sum_\mu \Df_\mu\kappa(\xi)^{m}X_\mu b(x)
\quad\mod\mathscr{S}^{m-2}_{1,0}.
$$
Since $b$ is arbitrary, the equality $\Df_\mu\kappa^m=m\kappa^{m-1}\Df_\mu\kappa\mod\mathscr{S}^{m-2}_{1,0}$ must hold.

(3) Compared to (2), the proof is quite elementary. We use induction on $n$. For $n=1$ the claim is proved by noticing that $|\nabla|^{-1}=\Op(\kappa)$ is a pseudo-differential operator of order $-1$, and the formula
$$
0=\Df_\mu\big(\kappa^{-1}\cdot\kappa\big)=\Df_\mu\kappa^{-1}\kappa+\kappa^{-1}\Df_\mu\kappa
+\sum_{\nu,\nu'} c_{\mu\nu\nu'}\Df_\nu\kappa^{-1}\Df_{\nu'}\kappa.
$$
For convenience we omit the dependence on $\xi$ for the quantities we consider below. Setting $K_n=\sum_{\mu}\|\Df_\mu\kappa^{-n}\|$, we obtain by the Leibniz property again
$$
K_{n+1}\leq K_n|\xi|^{-1}+C_1|\xi|^{-(n+2)}+C_2K_n|\xi|^{-2}.
$$
Inductively, this implies
$$
K_{n}\leq C_1(n+1)|\xi|^{-(n+1)}\big(1+C_2|\xi|^{-1}\big)^{n+1},
$$
Using the Leibniz property again, we find
$$
\Df_\mu\kappa^{-n}=\kappa^{-(n-1)}\Df_\mu\kappa^{-1}+\kappa^{-1}\Df_\mu\kappa^{-(n-1)}
+\sum_{\nu,\nu'} c_{\mu\nu\nu'}\Df_\nu\kappa^{-(n-1)}\Df_{\nu'}\kappa^{-1}.
$$
Inductively this gives (noting that by our choice, $c_{\mu\nu\nu'}$ is either 0 or 1)
$$
\begin{aligned}
\left\|\Df_\mu\kappa^{-n}-\frac{n\Df_\mu\kappa^{-1}}{\kappa^{n+1}}\right\|
&\lesssim \sum_{k=0}|\xi|^{k-2}K_{n-1-k}(\xi)\\
&\lesssim \frac{n^2\big(1+C|\xi|^{-1}\big)^{n+1}}{|\xi|^{n+2}}.
\end{aligned}
$$
Finally it suffices to notice that $\Df_\mu\kappa^{-1}=-\kappa^{-2}\Df_\mu\kappa$ plus a symbol of order $-3$.
\end{proof}

We then verify that the notion of order for quasi-homogeneous symbols in Definition \ref{QuasiHomoSym} coincides with Definition \ref{2Order}. Furthermore, the difference of a quasi-homogeneous symbol can be computed by an approximate Leibniz rule.

\begin{proposition}\label{PSSymbol}
Let $\kappa$ be the symbol of $|\nabla|$ on $\Gp$, $b$ be the symbol of a vector field on $\Gp$. Let $f$ be a bounded holomorphic function defined near $z=0$ on the complex plane, covering the closed disk of radius $R_0:=\sup_{x,\xi}|\xi|^{-1}\|b(x,\xi)\|$.

(1) The quasi-homogeneous symbol $f(b/\kappa)$ is of order 0 in the sense of Definition \ref{2Order}, and the difference $\Df_\mu f(b/\kappa)$ is such that
$$
\Df_\mu f\left(\frac{b(x,\xi)}{\kappa(\xi)}\right)
-\frac{1}{\kappa(\xi)}f'\left(\frac{b(x,\xi)}{\kappa(\xi)}\right)\Df_\mu b
-\frac{\Df_\mu\kappa(\xi)}{\kappa(\xi)^2}f'\left(\frac{b(x,\xi)}{\kappa(\xi)}\right)b(x,\xi)
$$
is a symbol of order $-2$.

(2) If $a$ is the symbol of another vector field, then the commutator $[f(b/\kappa),a]$ is still a symbol of order $0$.

(3) If $b$ has at least $C^1$ coefficients, then for any left-invariant vector field $X$, the symbol $Xf(b/\kappa)$ is of order $0$, and 
$$
Xf\left(\frac{b}{\kappa}\right)-\frac{1}{\kappa}f'\left(\frac{b}{\kappa}\right)Xb
$$
is a symbol of order $-1$.
\end{proposition}

\begin{corollary}\label{CoroPSSymbol}
Under the same assumptions of Proposition \ref{PSSymbol}, if in addition the vector field $\Op(b)$ is of class $C^r_*$ for $r>0$, then the symbol $\kappa^m f(b/\kappa)$ is of class $\mathcal{A}^m_r$.
\end{corollary}

Thus the symbolic calculus formulas are easier to manipulate with for para-differential operators corresponding to quasi-homogeneous symbols.

\begin{proof}[Proof of Proposition \ref{PSSymbol}]
The proof is basically a repeated application of majorant power series. There is a simple fact that we will keep using: if $b$ is the symbol of a vector field, then $\Df_\mu b$ is just a scalar function of $x$, hence commutes with any symbol. In fact, if $X$ is a left-invariant vector field, then $\Df_\mu \sigma[X]=(XQ_\mu)(e)$ by formula (\ref{Diff_qa}), which is a number. In the following, it obviously suffices to prove the estimates for representations $\xi$ which are sufficiently ``large" in the dominance order (i.e. $|\xi|$ sufficiently large).

(1) Suppose 
$$
f(z)=\sum_{n=0}^\infty c_nz^n,
$$
and set $R$ to be the radius of convergence of this power series. By assumption, $R_0=\sup_{x,\xi}|\xi|^{-1}\|b(x,\xi)\|<R$. 

It suffices to prove that the expression in the statement has operator norm bounded by $|\xi|^{-2}$, because it is in fact still a power series in $b/\kappa$ multiplied by $\{\Df_\nu b\}_\nu$ and $\Df_\mu\kappa$. For convenience we omit the dependence on $\xi$ below.

We first estimate the norm $B_n:=\sum_{\mu}\big\|\Df_\mu(b^n)\big\|$. By the Leibniz property we find
$$
B_n
\leq C_1(R_0|\xi|)^{n-1}+(R_0|\xi|+C_2)B_{n-1}
,
$$
so some elementary manipulation with geometric progression gives 
$$
B_n\lesssim n|\xi|^{-1}(R_0|\xi|+C_2)^{n}.
$$
The implicit constant does not depend on $\xi$. Similarly as in Proposition \ref{Dkappa}, we may use this estimate inductively to $\Df_\mu(b^n)$ to conclude
$$
\begin{aligned}
\big\|\Df_\mu(b^n)-nb^{n-1}\Df_\mu b\big\|
\lesssim n^2|\xi|^{-2}(R_0|\xi|+C_2)^{n}
\end{aligned}
$$
Combining this and Proppsition \ref{Dkappa}, we find that the operator norm of the quantity in the statement has a majorant series
$$
\frac{1}{|\xi|^2}\sum_{n=1}^\infty n^2|c_n|\frac{(R_0|\xi|+C_2)^{n}+(R_0|\xi|+C)^n}{|\xi|^n}.
$$
For $|\xi|$ sufficiently large, Cauchy estimate for derivatives of a holomorphic function shows that the series is absolutely convergent and sums to a bounded function. This proves the claim.

(2) It suffices to notice that $[a,b]=ab-ba$ is still the symbol of a vector field. Furthermore, using the Leibniz property of Lie brackets, we have
$$
[a,b^n]=[a,b]b^{n-1}+b[a,b^{n-1}],
$$
implying 
$$
\big\|[a,b^n]\big\|\lesssim n(R_0|\xi|)^{n}.
$$
The rest of the estimate is identical to that in (1).

(3) The proof is of the same spirit of (2): in fact,
$$
X(b^n)=nb^{n-1}Xb+\sum_{k=1}^{n-1}b^{n-k-1}[Xb,b^k].
$$
Note that $[Xb,b^k]$ is still the symbol of a classical differential operator, whose norm at $\xi$ is controlled by
$$
\big\|[Xb,b^k]\big\|\lesssim_{\Gp} |b|_{C^1}k(R_0|\xi|)^{k},
$$
as in (2). Here $|b|_{C^1}$ stands for the $C^1$ norm of the coefficients. Thus
$$
\big\|X(b^n)-nb^{n-1}Xb\big\|\lesssim n^2(R_0|\xi|)^{n-1}|b|_{C^1}
$$
and the majorant series argument is sufficient to imply the desired result.
\end{proof}

We conclude this subsection with a proposition regarding symbols that involve very rough coefficients.
\begin{proposition}\label{SymbolVeryRough}
Suppose $\sum_{n=0}^\infty c_nz^n$ is a convergent power series near $z=0$. Fix $r>0$. Let $s>n/2$, $s>r$. Let $a,b$ be symbols of vector fields on $\Gp$, with $a$ of merely $C^{-r}_*$ regularity in $x$ and $b$ of $H^s$ regularity in $x$, with
$$
\sup_{\xi\in\DuGp}\frac{\|b(x,\xi)\|_{H^s_x}}{|\xi|}
$$
suitably small. Define 
$$
[b]_na:=ab^n+bab^{n-1}+\cdots+b^na.
$$
Then 
$$
\sum_{n=0}^\infty c_n\kappa^{-n}[b]_na
$$
is a symbol of class $\mathcal{A}^{0}_{-r}(\Gp)$.
\end{proposition}
\begin{proof}
Suppose $c$ is any $H^s$ vector field on $\Gp$. Let us recall the para-product decomposition Theorem \ref{ParaPrd} to analyze the product $ac$. With a little abuse of notation, we write $ac$ to represent any coefficient appearing in the product symbol $ac$. Then
$$
ac=T_ac+T_ca+R(a,c).
$$
By Proposition \ref{T_aNegIndex}, $T_a$ is a para-differential operator of order $r$, so 
$$
|T_ac|_{C^{-r}_*}
\leq 
C\|T_ac\|_{H^{s-r}}
\leq C|a|_{C^{-r}_*}\|c\|_{H^{s}}.
$$
Since $T_c$ is a para-differential operator of order 0, $|T_ca|_{C^{-r}_*}\leq C|a|_{C^{-r}_*}\|c\|_{H^{s}}$. Finally, from the proof of Theorem \ref{ParaPrd}, the symbol $R[a]\in\mathscr{S}^{r}_{1,1}$, so by Stein's theorem,
$$
|R(a,c)|_{C^{-r}_*}
\leq 
C\|R(a,c)\|_{H^{s-r}}
\leq C|a|_{C^{-r}_*}\|c\|_{H^{s}}.
$$
To summarize, $|ac|_{C^{-r}_*}\leq C|a|_{C^{-r}_*}\|c\|_{H^{s}}$.

We now use this to analyze the product $b^kab^{n-k}$. Since $H^s$ is a Banach algebra, we have
$$
|ab^{n-k}|_{C^{-r}_*}
\leq C|a|_{C^{-r}_*}\|b^{n-k}\|_{H^{s}}
\leq C_s^{n-k}|a|_{C^{-r}_*}\|b\|_{H^{s}}^{n-k},
$$
which further implies
$$
|b^kab^{n-k}|_{C^{-r}_*}
\leq C_s^{n}|a|_{C^{-r}_*}\|b\|_{H^{s}}^{n}.
$$
Thus if $\sup_{\xi\in\DuGp}|\xi|^{-1}\|b(x,\xi)\|_{H^s_x}$ is sufficiently small, then the series of matrices
$$
\sum_{n=0}^\infty c_n\kappa^{-n}[b]_na
$$
will have an absolutely convergent majorant series with respect to the $C^{-r}_*$ norm. This concludes the proof.
\end{proof}

\section*{Acknowledgement}
The author would like to thank Professor Carlos Kenig for weekly discussion on this project, and Professor Gigliola Staffilani for constant support. The author benefits a lot from discussion with Professor Jean-Marc Delort, Veronique Fischer, Isabelle Gallagher and David Jerison. Thanks also goes to the author's friend, Kai Xu, for comments on representation theory.

\bibliographystyle{alpha}
\bibliography{References}

\end{spacing}
\end{document}